\documentclass[10pt,a4paper]{amsart}
\usepackage[draft]{optional}
\usepackage[british,american]{babel}

\usepackage{mathrsfs}
\usepackage{amsfonts, amsmath, wasysym}

\usepackage{amssymb,amsthm,
paralist
}

\usepackage{
latexsym,
nicefrac,
}
\usepackage[usenames]{color}

\usepackage{url}

\definecolor{darkgreen}{rgb}{0,0.5,0}
\definecolor{darkred}{rgb}{0.7,0,0}
\usepackage[colorlinks, 
citecolor=darkgreen, linkcolor=darkred
]{hyperref}
\usepackage{esint}
\usepackage{bibgerm}
\usepackage[normalem]{ulem}

\theoremstyle{plain}
\newtheorem{lemma}{Lemma}[section]
\newtheorem{thm}[lemma]{Theorem}

\newtheorem{cor}[lemma]{Corollary}

\theoremstyle{definition}
\newtheorem{defn}[lemma]{Definition}

\newtheorem{rmk}[lemma]{Remark}

\numberwithin{equation}{section}

\newcommand{\al}{\alpha}
\newcommand{\be}{\beta}

\newcommand{\de}{\delta}
\newcommand{\om}{\omega}
\newcommand{\Om}{\Omega}

\newcommand{\La}{\Lambda}

\newcommand{\si}{\sigma}

\newcommand{\Si}{\Sigma}
\renewcommand{\th}{\theta}

\newcommand{\ep}{\varepsilon}

\newcommand{\peps}{\partial_{\eps}}

\newcommand{\R}{\ensuremath{{\mathbb R}}}
\newcommand{\N}{\ensuremath{{\mathbb N}}}

\newcommand{\C}{\ensuremath{{\mathbb C}}}

\newcommand{\zlar}{{z_{\la}^{a,\bb}}}

\newcommand{\vlar}{{v_{\la}^{a,\bb}}}

\newcommand{\vlart}{{\tilde v_{\la}^{a,\bb}}}
\newcommand{\ularb}{\tilde \om_{\la}}

\newcommand{\jlar}{{j_{\la}^{a,\bb}}}

\newcommand{\omlrt}{{\tilde \bb_{\la}}}
\newcommand{\omlt}{{\tilde \bb_{\la}}}

\newcommand{\elart}{{e_{\la}^{a,\bb}}}

\newcommand{\ed}{{\rm d}}
\newcommand{\dd}{\,\ed}

\newcommand{\weakto}{\rightharpoonup}

\DeclareMathOperator{\inj}{inj}

\newcommand{\norm}[1]{\Vert#1\Vert} 

\def\osc{\mathop{{\mathrm{osc}}}\limits}

\newcommand{\beq}{\begin{equation}}
\newcommand{\eeq}{\end{equation}}
\newcommand{\beqs}{\begin{equation*}}
\newcommand{\eeqs}{\end{equation*}}
\newcommand{\beqa}{\begin{equation}\begin{aligned}}
\newcommand{\eeqa}{\end{aligned}\end{equation}}
\newcommand{\beqas}{\begin{equation*}\begin{aligned}}
\newcommand{\eeqas}{\end{aligned}\end{equation*}}
\newcommand{\brmk}{\begin{rmk}}
\newcommand{\ermk}{\end{rmk}}
\newcommand{\partref}[1]{\hbox{(\csname @roman\endcsname{\ref{#1}})}}
\newcommand{\half}{\frac{1}{2}}
\newcommand{\thalf}{\tfrac{1}{2}}

\newcommand{\leqs}{\lesssim}

\newcommand{\Pt}{P_{u_t}}

\newcommand{\ddt}{\tfrac{d}{dt}}

\newcommand{\abs}[1]{\vert#1\vert} 
\newcommand{\babs}[1]{\left\vert#1\right\vert}
 
\newcommand{\eps}{\varepsilon}
\newcommand{\na}{\nabla}

\newcommand{\TTom}{\norm{\tau_{g_{S^2}}(\om)}_{L^2(S^2)}}
\newcommand{\TTn}{\mathcal{T}_n}
\newcommand{\TT}{\mathcal{T}}
\newcommand{\bit}{\begin{itemize}}
\newcommand{\eit}{\end{itemize}}

\newcommand{\ddeps}{\tfrac{d}{d\eps}\vert_{\eps=0}}
\newcommand{\ddepsz}{\tfrac{d}{d\eps}}
\newcommand{\la}{\lambda}

\newcommand{\dist}{\text{dist}}

\newcommand{\ZZ}{\mathcal{Z}}
\newcommand{\VV}{\mathcal{V}}
\newcommand{\HH}{\mathcal{H}}

\newcommand{\Id}{\mathrm{Id}}

\newcommand{\Area}{\text{Area}}
\newcommand{\Loj}{{\L}ojasiewicz }
\newcommand{\Lojns}{{\L}ojasiewicz}
\newcommand{\bl}{\pi}

\newcommand{\UU}{\mathcal{U}}
\newcommand{\DD}{\mathbb{D}}
\newcommand{\innz}{\langle\cdot,\cdot \rangle_z}
\newcommand{\bbs}{{\hat \om}} 
\newcommand{\bbeps}{{\om^{(\eps)}}}
\newcommand{\bb}{\om} 

\newcommand{\Mob}{{\text{M\"ob}(S^2)}}
\newcommand{\HHz}{{\mathcal H_0}}
\newcommand{\etanew}{\eta_1} 
\newcommand{\VVMob}{{\VV_{\text{M\"ob}}}}
\newcommand{\VVzero}{{\VV_{0}}}
\newcommand{\NN}{\mathcal{N}}
\newcommand{\zz}{z}

\newcommand{\etaz}{{\eta_0}}
  \newcommand{\DDr}{{\DD_{\frac{r_0}{2}}}}
  \newcommand{\HHone}{\HH_1^{\si_1}(\bbs)}
   
\newcommand{\leps}{{\lambda_\eps}}
\newcommand{\jeps}{{j_\eps}}

\newcommand{\zeps}{{z_\eps}}

\newcommand{\omla}{{\tilde \om_{\la}}}

\newcommand{\tauS}{\tau_{g_{S^2}}}

\newcommand{\zone}{{z_{\la}^{a,\om}}}
\newcommand{\zbone}{{\hat z_{\la}^{b,\om}}}
\newcommand{\ztwo}{{z_{\tilde \la}^{\tilde a, \tilde \om}}}
\newcommand{\zbtwo}{{\hat z_{\tilde \la}^{\tilde b, \tilde \om}}}
\newcommand{\Ed}{{E_d}}

\newcommand{\rr}{{r_0}} 
\title[{\sc
{\L}ojasiewicz inequalities near simple bubble trees
}]{{\sc
{\L}ojasiewicz inequalities for almost harmonic maps near simple bubble trees
}}
\author{Melanie Rupflin}
\date{\today}

\begin{document}

\begin{abstract}
We prove \Loj inequalities for the harmonic map energy for maps from surfaces of positive genus into general analytic target manifolds which are close to simple bubble trees and as a consequence obtain new results on the convergence of harmonic map flow and on the energy spectrum of harmonic maps with small energy. 

Our results and techniques are not restricted to particular targets or to integrable settings and we are able to lift general \Lojns-Simon inequalities valid near harmonic maps $\bbs:S^2\to N$ to 
the singular setting whenever the 
bubble $\bbs$ is attached at a point which is not a branch point.
\end{abstract}
\maketitle
\section{Introduction}
Let $(\Si,g)$ be a closed orientable surface and let $(N,g_N)$ be a closed Riemannian manifold of any dimension, which by Nash's embedding theorem can be assumed to be  isometrically embedded $N\hookrightarrow \R^n$ in some Euclidean space. We recall that a map $u:\Si\to N$ is called a harmonic map if it is a critical point of the Dirichlet energy 
\beq
\label{eq:energy}
E(u):=\half \int_{\Si}\abs{du}^2 \ed v_g.
\eeq
Harmonic maps are characterised by $\tau_g(u)=0$, where the tension of $u:\Si\to N\hookrightarrow \R^n$ can be described as
$\tau_g(u)=P_{u}(\Delta_g u)=\Delta_g u+A(u)(\na u,\na u)$, $\Delta_g$ the Laplace-Beltrami operator of maps $u:(\Si,g)\to \R^n$ and $P_p:\R^n\to T_pN$ the orthogonal projection. Here and in the following $A(p)(v,w)=-(dP_p)(v)(w)$, $v,w\in T_pN$, denotes the second fundamental form of $N\hookrightarrow \R^N$ and we write for short 
$A(u)(\na u,\na u)=g^{ij}A(u)(\partial_{x_i} u,\partial_{x_j} u)$.

In the study of harmonic maps from closed surfaces of positive genus one is often confronted with the situation that the lowest possible energy level $E_0$ of homotopically non-trivial maps is not attained in the set of maps from the given surface $\Si$; instead minimising sequences may undergo bubbling and converge to a limiting configuration which is a simple bubble tree consisting of a trivial base map and a single bubble $\bbs$ given by a non-trivial harmonic map $\bbs:S^2\to N$. 
This singular behaviour means that the powerful techniques of \Loj inequalities as developed in the seminal work of Simon \cite{Simon} do not apply, even in the simplest such situation of degree one maps from the torus to $S^2$. As a result, questions such as 
 the discreteness of the energy spectrum near $E_0$ and the asymptotic behaviour of harmonic map flow for maps whose energy tends to $E_0$ are open in the setting of maps from higher genus surfaces.
 
The purpose of this paper is to address these and related questions not only in the special situation of maps to the sphere mentioned above, but more generally for maps into closed analytic manifolds of arbitrary dimension which are close to simple bubble trees for which the underlying bubble is attached at a non-branched point.

To this end, we first recall that the results of  \cite{Struwe1985, DT, QingTian, LinWang} imply that 
for any sequence of maps $u_n:\Si\to N$ with bounded energy and
$\norm{\tau_g(u_n)}_{L^2(\Si)}\to 0$ a subsequence converges
 strongly in $H^{2}_{loc}(\Si\setminus S)$ to a harmonic limit $u_\infty:\Si\to N$
 away from a finite set of points $S$ 
where a finite number of bubbles form and that in this convergence to a bubble tree there is no loss of 
energy and no formation of necks. 
 
If no bubbles form and if $N$ is analytic then we can apply 
 the work of Simon \cite{Simon} which establishes that  there exists a neighbourhood of $u_\infty$, a constant $C$ and an exponent $\gamma_{\infty}\in (1,2]$ so that for maps $u:\Si\to N$  in this neighbourhood the \Loj estimate
\beq
\label{est:Loj-Simon}
\abs{E(u)-E(u_\infty)}\leq C\norm{\tau_g(u)}_{L^2(\Si)}^{\gamma_{\infty}}\eeq
holds true. 
While the method of Simon from \cite{Simon} applies provided the maps are close to $u_\infty$ in $H^2$,  the above inequality is trivially satisfied for maps with bounded energy and large tension, so \eqref{est:Loj-Simon} holds whenever $u$ is $H^1$-close to $u_\infty$. 

However,  this result is not applicable for maps that undergo bubbling and  the only setting in which this problem has been overcome is 
in the major works \cite{Topping1997,Topping-quantisation} of Topping and \cite{Waldron} of Waldron on almost harmonic maps between spheres.
These results are based on 
a delicate analysis of almost harmonic maps which exploits in particular that for maps  between spheres
 the Dirichlet energy has a natural splitting into a holomophic and an antiholomophic part. 
  This allowed Topping \cite{Topping-quantisation}   
 to derive a \Loj estimate with optimal exponent 
\beq
\label{est:Loj-Peter-1}
\abs{E(u)-4k\pi}\leq C\norm{\tauS(u)}_{L^2(S^2)}^2
\eeq 
for maps between spheres which are close to large class of bubble trees and very recently for 
Waldron \cite{Waldron} to obtain \Loj estimates near general bubble trees.
 
Here we do not restrict our attention to a 
particular domain surface or a particular target but instead restrict the limiting configuration to the 
simplest situation where strong convergence fails, 
i.e. where the maps converge to a simple bubble tree consisting of a constant base map and a single bubble. 
In this situation the results of \cite{Struwe1985, DT, QingTian, LinWang} ensure that 
there exists a non-constant harmonic map $\bbs:S^2\to N$, points $a_n\to a$ and bubble scales $\la_n\to \infty$ so that, after passing to a subsequence, 
 $u_n\to \bbs(p^*)$ strongly in $H_{loc}^{2}(\Si\setminus \{a\})$ while on some fixed sized ball $B_r(a)$, working in local isothermal coordinates  $F_a:\Si \supset B_r(a)\to \DD_{\tilde r}\subset \R^2$, we have 
 \beq \label{eq:strong-conv-bubble-tree} u_n\circ F_a^{-1}-\bbs\circ \pi_{\la_n}^{F_a(a_n)}\to 0 \text{ strongly in } H^1(\DD_{\tilde r}) \cap L^\infty(\DD_{\tilde r}).
\eeq
Here $\pi_{\la}^b:=\pi(\la(x-b))$ for
$\pi:\R^2\to S^2\setminus \{p^*\}$ the inverse of the stereographic projection from the north pole $p^*=(0,0,1)^T$.
 
We note that despite this simple structure of the bubble tree the result of Simon \cite{Simon} is not applicable as we cannot view such maps as being
$H^1$-close to a critical point $u_\infty:\Si\to N$ of the energy. The only exception to this is when the domain is a sphere, as in this case we can modify any such sequence by suitable M\"obius transforms to obtain strong convergence to $\bbs$ on all of $S^2$. 
 For the rest of the paper we will thus assume that $\Si$ is a closed orientable surface of genus $\gamma\geq 1$. Since the energy is conformally invariant we can assume that our domain is either a flat unit area torus
or, for higher genus surfaces, that the metric $g$ is hyperbolic, i.e. has (Gauss)-curvature $-1$.

While also in the present work one of the key steps will be to relate the rate at which the tension tends to zero with the rate at which the bubble concentrate,  our method of proof will be very different to the ones in \cite{Topping-quantisation, Waldron}. In particular we will not require any information on the behaviour of general almost harmonic maps beyond the well known results on the bubble tree convergence recalled above.
 Instead our analysis will follow the approach developed in the joint work \cite{MRS} with Malchiodi and Sharp and we will derive our \Lojns-estimate by comparing maps $u:\Si\to N$ which undergo bubbling to maps in a specific finite dimensional set $\ZZ$ of what we call adapted bubbles. These adapted bubbles $\zz:\Si\to N$ provide models for maps converging to a simple bubble tree and are constructed so that the energy and its variations have the right properties on $\ZZ$.
 The key point of the method of proof is that 
a careful analysis of the energy and its variations on $\ZZ$ allows us to  obtain \Lojns-estimates for much more general almost critical points, without ever having to analyse such general almost critical points. We also refer the reader to Theorem 2.2 of \cite{MRS} which establishes \Lojns-estimates near (non-compact) finite dimensional manifolds of adapted critical points in the abstract setting of energies on Hilbert spaces, and to \cite{R-2bubbles} and \cite{R-rigidity} for recent applications of these ideas.

In the analysis of almost critical points of the $H$-surface energy in \cite{MRS} the set of bubbles is explicitly known, indeed consists of rotations of the identity, and the bubbles are non-degenerate critical points, i.e. so that the second variation of the energy is definite in directions orthogonal to the action of M\"obius-transforms. 

The present paper demonstrates that the ideas developed in \cite{MRS} can be applied to far more general settings, where neither of these simplifications is present. 
On the one hand, we shall not require any detailed information about the underlying bubbles $\bbs$. In particular, our proof does not rely on the explicit knowledge of the set of bubbles that for harmonic maps one would only have for special targets such as spheres. All we need to ask of the bubble is that it is attached to the base at a point that is not a branched point, i.e. that $d \bbs(p^*)\neq 0$, $p^*=\pi(\infty)=(0,0,1)$. 

Just as importantly, we shall see that our method does not rely on the non-degeneracy of the underlying critical point that is present in \cite{MRS} and we will be able to prove \Lojns-estimates even if $\bbs:S^2\to N$ is a harmonic map which has non-integrable Jacobi fields. Indeed we are able to 
 lift the  
\Lojns-Simon estimates \cite{Simon}
\beq
\label{est:Loj-S2}
\abs{E(\bbs)-E(\om)}\leq C\norm{\tauS(\om)}_{L^2(S^2)}^{\gamma_1}
\eeq
and 
\beq
\label{est:Loj-S2-dist}
\dist_{L^2}(\om, \{\tilde \om:S^2\to N \text{ harmonic }\})
\leq C \norm{\tauS(\om)}_{L^2(S^2)}^{\gamma_2},
\eeq
from the regular setting of maps $\om:S^2\to N$ which are close to $\hat \om$ to 
 obtain \Lojns-estimates 
with the same 
 exponents $\gamma_1\in (1,2]$ and $\gamma_2\in (0,1]$ in 
the singular setting of maps from $\Si$ which converge to simple bubble trees. To be more precise, we shall prove

\begin{thm}
\label{thm:bubbling-version}
Let $(\Si,g)$ be a closed oriented surface of positive genus and let $(N,g_N)$ be a closed analytic   manifold of any dimension. 
Let  
$(u_n)$ be a sequence of maps  with bounded energy which are almost harmonic in the sense that 
$$\TTn:=\norm{\tau_g(u_n)}_{L^2(\Si,g)}\to 0.$$
Suppose that $u_n$ converges as described above to a bubble tree consisting of a constant base map $u_\infty:\Si\to N$ and a single bubble $\bbs:S^2\to N$ which is so that 
$d \bbs(p^{*})\neq 0$.
Then, for sufficiently large $n$, we can bound the bubble scale $\la_n$ in  \eqref{eq:strong-conv-bubble-tree} by
 \beq
\label{claim:la-bubbling-thm}
\la_n^{-1}\leq  C\, \TTn\abs{\log\TTn}^{\half}
\eeq
and the difference in energy by 
\beq
\label{claim:Loj-2-bubbling-thm}
\abs{E(u_n,\Si)-E(\hat \omega,S^2)}\leq  C\,\TTn^{\, \gamma_1} \abs{\log\TTn}^{\frac{\gamma_1}{2}}
\eeq
for the same exponent 
$\gamma_1\in(1,2]$ for which \eqref{est:Loj-S2} holds near $\bbs:S^2\to N$.

Furthermore we can choose $\la_n\to \infty$, $a_n\to a$ and a sequence of harmonic maps $\om_n:S^2\to N$ which converge smoothly to $\om_\infty$ so that 
\beq
\label{claim:Loj-1-bubbling-thm}
\norm{\na \big( u_n-\omega_n \circ \pi_{\la_n}\circ F_{a_n}\big)}_{L^2(B_{r_1}(a))}+\norm{\na u_n}_{L^2(\Si\setminus B_{r_1}(a))}\leq C\TTn ^{\,\gamma_2} \abs{\log\TTn}^{\frac{\gamma_2}2}
\eeq
and 
\beq
\label{claim:L2-est-thm-1}
\norm{u_n-\omega_n(p^*)}_{L^2(\Si,g)}\leq C \TTn ^{\,\gamma_2}\abs{\log\TTn}^{\frac{\gamma_2}2}+C\TTn \abs{\log\TTn}
\eeq
and so that for every $r>0$ there exists a constant $C$ with
\beq\label{claim:L2-est-thm-2}
\norm{u_n\circ (\pi_{\la_n}\circ F_{a_n})^{-1}-\omega_n}_{L^2(S^2\setminus B_{r}(p^*))}\leq C \TTn ^{\,\gamma_2}\abs{\log\TTn}^{\frac{\gamma_2}2}.
\eeq
Here $\gamma_2\in (0,1]$ is the same exponent for which \eqref{est:Loj-S2-dist} holds,  
$F_{a_n}$ are local isothermal coordinates centred at $a_n$ as introduced in Remark \ref{rmk:Fa} and $r_1>0$ is a fixed radius.
\end{thm}

\begin{rmk}
If all Jacobi fields along $\bbs$ are integrable then we can drop the assumption that $N$ is analytic and obtain the above result for  $\gamma_1=2$ and $\gamma_2=1$. However, as observed by Eells and Wood in  \cite{Lemaire-Wood2} even energy minimisers can have non-integrable Jacobifields. Conversely the works of Gulliver and White \cite{Gulliver-White}
and Lemaire and Wood \cite{Lemaire-Wood1} establish that all Jacobi-fields along harmonic spheres are integrable if the target is homotopic to $S^2$ or $\C P^2$.
\end{rmk}

Over the past decades \Lojns-estimates have become a well established tool in the analysis of variational problems in non-singular settings. However to date there are few instances of \Lojns-estimates in settings with singularities or with a change of topology. In addition to \cite{Topping-quantisation} and \cite{MRS} mentioned above, such results were obtained in the major papers of Colding-Minicozzi \cite{Colding-Minicozzi} and Chodosh-Schulze \cite{Chodosh-Schulze} on the uniqueness of blow-ups of Mean Curvature flow and by Glaudo-Figalli \cite{Glaudo-Figalli}  and Deng-Sun-Wei \cite{Deng-Sun-Wei} on critical points of the Sobolev-inequality.
One of the reasons that 
\Lojns-estimates have attracted a lot of interest is their versatility in applications both to variational problems and to the analysis of evolution equations. They can be used in particular to establish convergence of gradient flows as well as to analyse the energy spectrum of critical points. As a consequence of Theorem \ref{thm:bubbling-version} we will hence obtain new results both on the asymptotic behaviour of harmonic map flow 
\beq
\label{eq:HMF}
\partial_t u=-\na^{L^2}E(u)=\tau_g(u), \qquad u(t=0)=u_0\in H^1(\Si,N),
\eeq
as well as on the energy spectrum of harmonic maps from higher genus surfaces into general analytic manifolds.

Simon's results \cite{Simon} imply that the energy spectrum $\{E(u): u:S^2\to N \text{ harmonic}\}$
of harmonic maps from $S^2$ into any analytic manifold $N$
is discrete below the level $2E_{S^2}$,
$$E_{S^2}:=\min\{E(u): u:S^2\to (N,g_N) \text{ harmonic, non-constant}\},$$
since harmonic maps with energy $E(u_n)\to E_\infty< 2E_{S^2}$ can always be pulled-back by suitable M\"obius transforms to ensure that they subconverge strongly.

Conversely, for surfaces of positive genus, \cite{Simon} only implies that the energy spectrum $\{E(u): u:(\Si,g)\to (N,g_N) \text{ harmonic}\}$ is discrete below the energy level $E_{S^2}$. 
Theorem \ref{thm:bubbling-version} now allows us to deduce the following result, which is in particular of interest for maps into three-manifolds, where the 
 results \cite{G-O-R} of Gulliver, Osserman and Royden ensure that area minimising 
surfaces cannot have true branch points. 

\begin{cor} \label{thm:energy-spectrum}
Let $(N,g_N)$ be a closed analytic manifold of any dimension and let  $(\Si,g)$ 
be a closed surface of positive genus. 
Then  the energy spectrum of harmonic maps from $(\Si,g)$ to $N$ below the level 
\beq
\label{def:Estar}
E^*:= \min( 2E_{S^2}, E_{(\Si,g)}+E_{S^2}, E_{S^2}^*)\eeq
 is discrete, where
 \beqa
 E_{(\Si,g)}&:=\inf\{E(u): u:(\Si,g)\to N \text { harmonic, non-constant}\}\\
 E_{S^2}^*&:=\inf\{E(\om): \om:S^2\to N \text{ branched, harmonic, non-constant}\}.
   \eeqa

\end{cor}

To state our results on harmonic map flow, we first recall that the work of Struwe \cite{Struwe1985} establishes the existence of a global weak solution of \eqref{eq:HMF} which has non-increasing energy and which is smooth away from finitely many times at which bubbling occurs. While solutions of this flow always subconverge along a sequence of times $t_j\to \infty$ either to a harmonic map or 
to a bubble tree of harmonic maps,  Topping \cite{Topping1997} showed that one cannot expect that the whole flow converges as $t\to \infty$ for general smooth target manifolds.
Conversely it is conjectured that for analytic targets the flow must indeed converge.
 If no bubbling occurs at infinite time this already follows from the work of Simon \cite{Simon}, while for maps from $S^2$ to $S^2$ the \Loj inequalities of Topping \cite{Topping-quantisation} and Waldron \cite{Waldron} yield convergence results for the flow.

We can now establish convergence of harmonic map flow into any analytic target manifold $(N,g_N)$  provided the initial energy is below the above mentioned energy threshold. We stress that this constraint on the energy does allow for bubbling, though restricts the potential limiting configurations to the simple bubble trees considered in Theorem \ref{thm:bubbling-version}. 

\begin{thm}\label{thm:flow}
Let $(N,g)$ be a closed analytic manifold of any dimension and let $u_0\in H^1(\Si,N)$ be any map with $E(u_0)\leq  E^*$ for $E^*$ defined in 
\eqref{def:Estar}. 
Then the corresponding solution of harmonic map flow \eqref{eq:HMF} either converges smoothly to a harmonic map $u_\infty:\Si\to N$ as $t\to \infty$ as described in \cite{Simon} or it converges to a simple bubble tree in the following sense:

There 
 exists a point $a\in \Si$ and a harmonic sphere $\hat \om:S^2\to N$ so that 
the energy of $u(t)$ converges to $E_\infty=E(\hat \om)$ 
at a rate of
\beq
\label{claim:delta-energy-flow-sphere}
\abs{E(u(t))-E_\infty}\leq C e^{-c_1\sqrt{t}}, \quad c_1=c_1(N,(\Si,g),E_\infty)>0
\eeq 
if $\gamma_1=2$ respectively, if $\gamma_1\in (1,2)$, at a rate of 
\beq
\label{claim:delta-energy-flow-sphere-degenerate}
\abs{E(u(t))-E_\infty}\leq C t^{-\frac{\gamma_1}{2-\gamma_1}}(\log t)^{\frac{\gamma_1}{2-\gamma_1}},
\eeq
while for any $\al<\frac{\gamma_1-1}{\gamma_1}$
the maps converge in $L^2$ at a rate of
\beq \label{est:L^2-est-flow}
\norm{u(t)-\om(p^*)}_{L^2(\Si)}\leq 
C\abs{E(u(t))-E_\infty}^{\al},
\eeq
as well as in $C^k$ on every compact subset $K$ of $\Si\setminus \{p^*\}$ also at a rate of
\beq 
\label{est:C^k-est-flow}
\norm{u(t)-\om(p^*)}_{C^k(K)}
\leq C\abs{E(u(t))-E_\infty}^{\al}.
\eeq
Here $\gamma_1\in (1,2]$ is so that \eqref{est:Loj-S2} is valid with exponent $\gamma_1$ for all harmonic spheres with energy $E_\infty$ and 
the constant $C$ is allowed to depend on the setting, the specific solution and, in case of 
 \eqref{est:L^2-est-flow} and \eqref{est:C^k-est-flow}, additionally on $\al- \frac{\gamma_1-1}{\gamma_1}>0$ and $K$. 
\end{thm}

\begin{rmk}
As the set of harmonic spheres with energy $E_\infty<2E_{S^2}$ is compact modulo M\"obius transforms, there always exists an exponent $\gamma_1\in(1,2]$ so that \eqref{est:Loj-S2} holds true 
for any harmonic sphere $\hat \om$ with $E(\hat \om)=E_{\infty}$. 
If all Jacobifields along harmonic maps of energy $E_\infty$ are integrable then we can drop the assumption that $N$ is analytic and  choose $\gamma_1=2$.
\end{rmk}
The first setting in which it was known that harmonic map flow must become singular, be it at finite or infinite time, is for degree $\pm 1$ maps $u_0$ from the torus to the sphere where the results of Eells-Wood 
 \cite{Eells-Wood} exclude the existence of a harmonic map that is homotopic to $u_0$. While $E^*=8\pi$ for $N=S^2$, in this setting we obtain the following improvement of the above result. 
\begin{cor}
\label{cor:torus}
Suppose that $u_0\in H^1(T^2,S^2)$ has degree $\pm1$ and energy $E(u_0)<12\pi$. Then the corresponding solution of harmonic map flow \eqref{eq:HMF} either develops a bubble at a finite time $T$ after which the flow converges exponentially to a constant or the flow converges to a simple bubble tree at a rate of $O(e^{-c\sqrt{t}})$ as described in  Theorem \ref{thm:flow}. 
\end{cor}

The paper is organised as follows: \\
In Section \ref{sec:def-bubble-set} we explain the construction of the adapted bubbles $\zz:\Si\to N$ with which we will later compare more general almost harmonic maps $u:\Si\to N$
and state a version of our \Lojns-estimates for maps in a uniform $H^1\cap L^\infty$ neighbourhood of the resulting finite dimensional manifold $\ZZ$, compare Theorem \ref{thm:main}. This theorem will then be proved in the subsequent Sections \ref{sec:key-lemmas} and \ref{sect:proof-main} and will in turn form the basis of the proof of Theorem \ref{thm:bubbling-version} and of all other main results of the paper.

\section{Definition of the adapted bubbles}
\label{sec:def-bubble-set}

Our set of adapted bubbles will be a finite dimensional manifold of maps
\beq
\label{def:Z}
\ZZ:=\{z_{\la}^{a,\om}:\Si\to N, \la\geq \la_1, a\in \Si, \om\in \HH_1^{\si_1}(\bbs)\}
\eeq
obtained by scaling maps $\om:S^2\to N$, which are elements of a suitable finite dimensional manifold $\HH_1^{\si_1}(\bbs)$, with a large factor $\la$ and then gluing them in a specific way to a point $a\in \Si$ 
as we describe in detail in the second part of this section.  

A crucial point in the construction of this manifold $\ZZ$ is to ensure that the second variation of the energy is uniformly definite orthogonal to $\ZZ$. Therefore the choice of the set of the underlying maps $\HH_1^{\si_1}(\bbs)$ from $S^2$ to $N$, which we use to define the elements of $\ZZ$, will crucially depend on the properties of the second variation of the energy at the limiting harmonic sphere $\bbs:S^2\to N$.

We recall that  $w\in \Gamma(\hat \om^*TN)$ is called a Jacobi-field along $\bbs$ if 
$d^2E(\bbs)(w,v)=0 \text{ for all } v\in \Gamma(\bbs^*TN), $
or equivalently if $w$ is a solution of 
\beq
\label{def:Jacobi-2}
L_\bbs(w):= P_\bbs(\ddeps \tau(\pi_N(\bbs+\eps w))=0.
\eeq
As the tension transforms according to 
$
\tau(\bb \circ q)=\half \abs{\na q}^2 \tau(\bb)  \circ q
$
under conformal changes $q$, we know that any variation $M^{(\eps)}$ of $\Id:S^2\to S^2$  in the set of M\"obius transforms $\Mob$ induces a Jacobifield $w=\ddeps ( \bbs \circ M^{(\eps)})$ along $\bbs$.

If the
second variation of the energy is non-degenerate at $\bbs$ in the sense that all Jacobi fields are of this form, as was the case in \cite{MRS}, then
we set
$$\HH_1^{\si_1} (\bbs):=\{ \bbs\circ R: R \in SO(3), \abs{p^*-Rp^*}\leq \si_1 \}$$
for a sufficiently small number $\si_1>0$.

If all Jacobifields along $\bbs$ are integrable, 
but not necessarily induced by M\"obius-transforms, 
then we use that the set 
 of harmonic maps near $\bbs$
is a manifold $\HH(\bbs)$ with $T_{\bbs}\HH(\bbs)=
\ker(L_{\bbs})$
 on which the energy is constant, compare \cite{Simon}. 
In this situation we can split 
\beq\label{eq:split-ker}
\ker(L_{\bbs})=\VV_0(\bbs)\oplus \VVMob(\bbs), \quad 
\VVMob(\bbs):= T_{\bbs} \{\bbs \circ M: M\in \Mob\}\eeq
$L^2$-orthogonally, fix a parametrisation $\Psi_1: \ker(L_{\bbs}) \supset \UU\to \HH(\bbs)$ with $\Psi_1(0)=\bbs$ and $d\Psi_1(0)=\Id$ 
and consider the submanifold $\HHz(\bbs)= \Psi_1(\UU\cap \VV_0(\bbs))$ which, for $\UU$ small, is transversal to the action of M\"obiustransforms. 
 For suitably small $\si_1>0$ we then let
$\HH_0^{\si_1}(\bbs):= \{\Psi_1(w): w\in \VV_0(\bbs) \text{ with } \norm{w}_{L^2}\leq \si_1\}$
 and set  
\beq
\label{def:H1} 
\HH_1^{\si_1}(\bbs):=\{ \bb \circ R: R\in SO(3), \bb\in \HH_0^{\si_1}(\bbs) \text{ with }
\abs{p^*-Rp^*}\leq \si_1
\}.\eeq
If we are instead dealing with a non-integrable setting, then we also need to
consider maps that are obtained by adapting certain non-harmonic maps $\om:S^2\to S^2$ in order to obtain a  set of adapted bubbles  which is large enough to capture all non-definite directions of the second variation. 
In this situation we choose  $\HHz(\bbs)$ as a suitable submanifold of the manifold used in the paper \cite{Simon} of Simon and define 
$\HH_0^{\si_1}(\bbs)$ and  
$\HH_1^{\si_1}(\bbs)$ as described above. 
 We discuss the precise definition of $\HHz(\bbs)$ in this case in
Appendix \ref{sec:Simon} and for now simply record that it has the following properties:

\begin{lemma}\label{lemma:Simon}
Let $\bbs:S^2\to N$ be a harmonic map into an analytic target $N$, let $k\in\N$, $\beta>0$ and let 
$\HH_0(\bbs)$ be the submanifold of $C^{k+2,\beta}(S^2,N)$
defined in Appendix \ref{sec:Simon}. Then 
\beq
\label{claim:split-deg}
\ker(L_{\bbs})=T_\bbs \{\bb\circ M: \bb \in \HHz(\bbs), M\in \Mob \}\eeq 
and there
exists a constant $C$ so that for any $\om\in \HHz(\bbs)$
\beq \label{claim:Ck-equiv}
\norm{\tau_{g_{S^2}}(\om)}_{C^k(S^2)}\leq C\norm{\tau_{g_{S^2}}(\om)}_{L^2(S^2)}
\eeq
 and so that we can choose $\bbeps$ in $\HHz(\bbs)$ with 
$\om^{(\eps=0)}=\om$,  $\norm{\peps \bbeps}_{C^k(S^2)}\leq C$ and
\beq \label{claim:tang-part-tension-below}
\tfrac{d}{d\eps} E(\bbeps)\geq \norm{\tau_{g_{S^2}}(\om)}_{L^2(S^2)}.
\eeq
\end{lemma}

Here and in the following all derivatives with respect to $\eps$ will be evaluated at $\eps=0$. 

Having thus chosen the set $\HH_1^{\si_1}(\bbs)$ of maps we want to scale and glue to a point $a\in\Si$, we now turn to the precise construction of the maps $z_{\la}^{a,\om}$, for $a\in \Si$, $\om \in \HH_1^{\si_1}(\bbs)$  and sufficiently large $\la$. 
We note that the right definition of these maps $z_{\la}^{a,\om}$ is crucial to ensure that the first and second variation of the energy on $\ZZ$ have the right properties for our method of proof to work, compare also \cite[Theorem 2.2]{MRS}.

Let $ \bl:\R^2\to S^2\setminus \{p^*\}$, $p^*:=(0,0,1)^T$, be the inverse stereographic projection
\beq 
\bl(x)=\big(\frac{2x}{1+\abs{x}^2}, \frac{\abs{x}^2-1}{\abs{x}^2+1}\big) \eeq
and set $\bl_\la(x):=\bl(\la x)$, $\la>0$. 
We want to define our adapted bubbles 
$$\zlar:\Si\to N \quad \text{ for } \om\in \HH_1^{\si_1}(\bbs), \la\geq \la_1, a\in \Si,$$ $\la_1$, $\si_1$ chosen later,
 in a way that $\zlar(p)\approx \omega(\bl_{\la}(x))$ in the following local isothermal coordinates $x=F_a(p)$.

\begin{rmk}\label{rmk:Fa}
Given any $a\in \Si$ we let 
 $F_a:B_\iota(a)\to \DD_{\rr}=\{x\in \R^2: \abs{x}<\rr\}$, $\iota:=\half\inj(\Si,g)$,
be as in Remark 3.1 of \cite{MRS}: 
 If $(\Si,g)$ is a flat unit area torus we set  $\rr=\iota$ and
use  Euclidean translations $F_a$ to the origin on a fundamental domain as  coordinates.
Conversely,  for higher genus surfaces we set 
$\rr= \tanh(\iota/2)$ and choose an orientation preserving isometric isomorphism 
$F_a$ that maps $(B_{\iota}(a),g)$ to the disc $\mathbb{D}_{\rr}$ in the Poincar\'e hyperbolic disc 
$(\mathbb{D}_{1},\tfrac{4}{(1-|x|^2)^2}g_E)$.
\end{rmk}
While in the higher genus case this only determines the maps $F_a$  upto a rotation of the domain, the specific choice of $F_a$ will not affect the definition of the set of adapted bubbles as  rotations of the coordinates correspond to the action of the subgroup of $SO(3)$ which fixes $p^*$. In the few places where a consistent choice of $F_a$ for $a$ in a neighbourhood of some $a_0$ is needed, we can fix a tiling of the Poincar\'e hyperbolic disc and use hyperbolic translations to the origin as explained in \cite[Remark 3.4]{MRS}.

Since 
\beq \label{eq:expansion-pi}
\bl_\la(x)=p^*+(\tfrac{2x}{\la \abs{x}^2},0)^{T}+O(\la^{-2}) \text{ for } \abs{x}\geq c>0\eeq
we can write 
\beq \label{eq:expansion-uR}
\tilde \bb_\la(x):= \bb (\pi_\la(x))=\bb(p^*)+d\bb(p^*) 
(\tfrac{2x}{\la \abs{x}^2},0)^{T}+O(\la^{-2}) \text{ for } \abs{x}\geq c>0
\eeq
  where we note that this expansion is valid for the function $ \tilde \bb_\la$ as well as its derivatives with respect to $x$. 
We shall later on consider variations $z_\eps$ of adapted bubbles obtained by variations of either the bubble parameter $\la_\eps$ or of the underlying map $\om^{(\eps)}\in \HHone$ and will always assume that these variations are chosen so that 
\beq
\label{ass:var}
\abs{\peps \la_\eps}\leq C \la \text{ and } \norm{\peps \om^{(\eps)}}_{C^2(S^2)}\leq C
\eeq 
  as this corresponds to variations of order $1$ after rescaling. 
 We note that for such variations the expansion \eqref{eq:expansion-uR} gives also an expansion for $\peps \tilde \bb_{\la_\eps}^{(\eps)}$ 
and its spatial derivatives with an error term of the same order 
$O(\la^{-2})$.

As in \cite{MRS} we  modify $\tilde \bb_\la$  with the help of the Green's function $G$, characterised by 
\beq 
\label{def:G} 
-\Delta_p G(p,a)=2\pi \de_a-2\pi (\Area(\Sigma,g))^{-1} \text{ on } \Si.\eeq
Letting $G_a$ be the function that represents $G$ in the above coordinates we recall that 
\begin{equation}\label{eq:locG}
G_a(x,y):=G(F_a^{-1}(x), F_a^{-1}(y))= - \log |x-y| + J_a(x,y), \quad x,y\in \DD_{r_0}
\end{equation}
for a smooth harmonic function $J_a$ which represents the regular part of Green's function, see \cite{MRS} for more detail. In particular 
\beq
\label{eq:locG-deriv}
 \na_y G_a(x,0)=\frac{x}{\abs{x}^2}+ \na_y J_a(x,0)
\eeq
and we will use this to adapt the maps   $\tilde \bb_\la$ to give well defined maps $v_{\la}^{a,\om}:\Si\to \R^n$, which we will later project onto $N$ to obtain our adapted bubbles  $\zlar:\Si\to N\hookrightarrow\R^n$. 
To do this, we let $j_{\la}^{a,\bb}:\DD_{r_0}\to \R^n$ be defined by 
\beq \label{def:jla}
j_{\la}^{a,\bb}(x):=\tfrac{2 }{\la}d\bb(p^*)\big( \na_y J_a(x,0)-\na_y J_a(0,0),0\big)^T,
\eeq 
and fix a cut-off function $\phi\in C_c^\infty(\DD_\rr,[0,1])$ with $\phi\equiv 1$ on $\DD_{\frac \rr 2}$, $r_0>0$ as in Remark \ref{rmk:Fa}.  
The maps $v_{\la}^{a,\om}:\Si\to \R^n$ are then defined as 
\beq \label{def:v-away}
v_{\la}^{a,\om}(p) = 
 \omega(p^*)+ 
\tfrac{2 }{\la}d\bb(p^*)
\big(\partial_{a^1}G(p,a)-\partial_{y_1}J_a(0,0) , \partial_{a^2} G(p,a)- \partial_{y_2}J_a(0,0),0\big)^T
\eeq on $\Si\setminus B_{\iota}(a)$	
where 
 $\partial_{a^i} = (F_a^{-1})_{\ast} \partial_{y^i}$, while on 
  $ B_{\iota}(a)$ we set
$v_{\la}^{a,\om}(p) = \tilde v_{\la}^{a,\om}(F_a(p))$ 
for $\tilde v_{\la}^{a,\om}:\DD_\rr\to \R^n$ given by 
\beq 
\label{def:tilde-v}
\tilde v_{\la}^{a,\om}:= \phi\big[ \ularb+ j_{\la}^{a,\bb}\big] + (1-\phi)
\big[ \bb(p^*)+
\tfrac{2 }{\la}d\bb(p^*)\big( \na_y G_a(\cdot,0)-\na_y J_a(0,0),0\big)^T].
\eeq
We note that for $N=S^2\hookrightarrow \R^3$ and $\omega=\text{Id}$ this definition of $\vlar$ essentially agrees with the choice of the adapted bubbles in the $H$-surface case in \cite{MRS} except that here we need to ensure that $\jlar(0)=0$ as our problem does not have the translation invariance present in \cite{MRS}.

On $B_{\iota}(a)$  we can use that the function $\vlar$ is represented in the above coordinates by 
\beqa
\vlart=\ularb+j_{\la}^{a,\bb}+\elart
\label{eq:writing-v-lart}
\eeqa
for an error term $\elart$ that is supported on $\DD_{r_0}\setminus \DD_{\frac{r_0}{2}}$ and there of order 
\beqa 
\label{est:elart} 
\norm{\elart}_{C^2}+ \norm{\partial_\eps \elart}_{C^2}
=O(\la^{-2}) .
\eeqa
We will in particular use that 
since $j_{\la}^{a,\bb}(0)=0$ we have 
\beq \label{est:v-hat-u-ball}
\abs{ \vlart(x)-\ularb(x)}\leq C\la^{-1}\abs{x} +C\la^{-2} \text{ and } \quad \abs{\na (\vlart-\ularb )}\leq C\la^{-1}
\text{ on } 
\DD_{r_0}
\eeq 
and that the analogue estimates also hold true for the derivatives of these quantities with respect to $\eps$.
Away from 
$B_{\tilde \iota}(a):=F_a^{-1}(\DD_{\rr/2})$, we can instead use that
\beq
\label{est:v-away}
\norm{\vlar-\bb(p^*)}_{C^2(\Si\setminus B_{\tilde \iota}(a))}+
 \norm{\partial_\eps (\vlar-\bbeps(p^*))}_{C^2(\Si\setminus B_{\tilde \iota}(a))}
\leq C \la^{-1}.
\eeq 

We now let $\de_N>0$ be so that the nearest point projection $\pi_N$ to $N$ is well-defined and smooth in a $\de_N$-tubular neighbourhood of $N\hookrightarrow \R^n$. 
Then, for sufficiently large
 $\la_1\geq 2$, the above estimates imply that 
 $\dist(\vlar(\cdot),N)\leq C\la_1^{-1}<\de_N$ on $\Si$
 allowing us to project these maps to define our adapted bubbles $\zlar:\Si\to N$ by
\beq 
\label{def:z-bubble} 
\zlar(p):= \pi_N(v_{\la}^{a,\om}), \qquad \la\geq\la_1, a\in \Si, \om\in \HHone.
\eeq

 \begin{rmk} \label{rmk:conventions}
In the following all results are to be understood as being true for
the set of adapted bubbles $\ZZ=\ZZ_{\la_1}^{\si_1}$ for
 sufficiently large $\la_1\geq 2$ and sufficiently small $\si_1>0$, both allowed to depend only on $(\Si,g)$, $N$ and $\bbs$. 
At times we will furthermore need to consider a smaller subset of this set given by  
\beq
\label{def:Z-subset} 
\ZZ_{\bar \la}^{\bar \si}:= \{z_{\la}^{a,\om}: \la\geq \bar \la, a\in \Si \text{ and } \om\in \HH_1^{\bar \si}(\hat \om)\}
\eeq
for suitable $\bar \la\geq \la_1$ and $0<\bar \si\leq \si_1$.
Furthermore we use the convention that $C$ denotes a constant, allowed to change from line to line, which only depends on $\bbs$, $(\Si,g)$ and $N$ unless indicated otherwise and we will use the shorthand $A\leqs B$ to mean that $A\leq CB$ for such a constant $C$. 
 \end{rmk}

In the following it will be important that we do not work with respect to the standard inner product on $H^1(\Si,\R^n)$, but instead use an inner product that appropriately weighs the $L^2$-part of the norm in the bubble region.

\begin{defn}\label{def:inner-product}
Given any $z=\zlar\in\ZZ$ we consider the
inner product
\beq\label{eq:def-inner-product}
\langle v,w\rangle_z :=\int_{\Si} \na v \na w + \rho_{z}^2 v w  \dd v_g,\quad v,w\in H^1(\Si,\R^n)
\eeq
where the weight $\rho_z$ is given by $\rho_z\equiv \frac{\la}{1+\la^2r_0^2}$ on $\Si\setminus B_\iota(a)$ while
\beqs 
\rho_z(p)=\rho_\la(x):=\frac{1}{2\sqrt{2}}\abs{\na \pi_\la(x)} =
\frac{\la}{1+\la^2\abs{x}^2}\text{ for } p=F_a^{-1}(x)\in B_\iota(a).
\eeqs 
\end{defn}
At times we will also want to use local versions  of the above norm so set
\beq
\label{def:norm-local}
\norm{w}_{z,\Om}^2:= \norm{\na w}_{L^2(\Om)}^2+\norm{\rho_z w}_{L^2(\Om)}^2, \quad \text{ for } \Om\subset \Si.
\eeq 
 We note that the weight $\rho_z:\Si\to \R^+$ is continuous and that we can bound  
 \beq 
 \label{est:z-by-weight}
 \abs{\na z}\leq C\rho_z \text{ and } \norm{\na (P_z w)}_{z}\leq C\norm{w}_{z}
 \eeq 
  for any $w\in H^1(\Si, \R^n)$ and any $z\in \ZZ$.
 A short calculation, see Appendix \ref{appendix:technical}, gives
 \beq
 \label{est:MV} \babs{\fint_{\Si}w \dd v_g}\leq C_{(\Si,g)}(\log\la)^\half \norm{ w}_{z}
 \eeq
  and thus allows us to bound
 \beq
\label{est:Lp-by-z-norm}
\norm{w}_{L^p(\Si,g)}\leq C(\log\la)^\half\norm{w}_{z},  \qquad p\in [1,\infty),
\quad C=C(p,(\Si,g)).
\eeq

With these definitions in place we can finally formulate our main result in the form that we shall prove in Sections \ref{sec:key-lemmas} and \ref{sect:proof-main} and that will subsequently form the basis of the proofs of all other main results.  

\begin{thm}\label{thm:main}
Let $(\Si,g)$ be any closed surface of positive genus, let $N$ be any analytic  closed  manifold and let $\bbs:S^2\to N$ be any harmonic map with $d\bbs(p^*)\neq 0$. 
Then there exist numbers $\eps>0$, $\bar\la\geq \la_1$, $\bar \si\in (0,\si_1)$ and $C<\infty$ 
so that for 
every $u\in H^1(\Si,N)$ for which there exists $\tilde z\in\ZZ_{\bar \la}^{\bar \si}$ with
$$\norm{\na(u-\tilde z)}_{L^2(\Si,g)}+\norm{u-\tilde z}_{L^\infty(\Si,g)}<\eps,$$
 $\ZZ_{\bar \la}^{\bar\si}$ 
 as in Remark \ref{rmk:conventions}, 
we can bound
\beq
\label{claim:Loj-1}
\dist(u,\ZZ):=\inf_{z\in\ZZ}\norm{u-z}_{z}\leq C\norm{\tau_g(u)}_{L^2(\Si,g)} (1+\abs{\log\norm{\tau_g(u)}_{L^2(\Si,g)}}^\half)
\eeq
while the estimate 
\beq
\label{claim:Loj-2}
\abs{E(\hat \omega)-E(u)}\leq C \norm{\tau_g(u)}_{L^2(\Si,g)}^{\gamma_1} (1+\abs{\log{\norm{\tau_g(u)}_{L^2(\Si,g)}}}^\half)^{\gamma_1},
\eeq
holds true for the exponent 
$\gamma_1\in (1,2]$ for which \eqref{est:Loj-S2} holds.

Furthermore, for each such $u$ there exists $z\in\ZZ$ with 
 $ \norm{u-z}_{z}=\dist(u,\ZZ)$ 
and the bubble scale of any such $z=z_{\la}^{a,\om}$ satisfies
\beq \label{claim:lambda}
\la^{-1}\leq C\norm{\tau_g(u)}_{L^2(\Si,g)}\cdot (1+ \abs{\log{\norm{\tau_g(u)}_{L^2(\Si,g)}}}^{\half})
\eeq 
while the tension of the underlying map $\om:S^2\to N$ is controlled by
\beq
\label{claim:tension}
\norm{\tau_{g_{S^2}}(\om)}_{C^2(S^2)}\leq C\norm{\tau_g(u)}_{L^2(\Si,g)}\cdot (1+ \abs{\log{\norm{\tau_g(u)}_{L^2(\Si,g)}}}^\half).
\eeq
\end{thm}

\section{Properties of the energy on the set of adapted bubbles}
\label{sec:key-lemmas}
\subsection{Basic properties of the second variation of the Dirichlet energy} $ $\\
We first recall the following standard expression for the second variation of the Dirichlet energy for which we include a short proof for the convenience of the reader.
\begin{lemma}
\label{lemma:2ndvar}
For any $u\in H^1(\Si,N)$ and $v,w\in \Gamma^{H^1\cap L^\infty}(u^*TN)$ we can write the second variation of the Dirichlet energy 
$d^2E(u)(v,w):=\tfrac{d}{d\eps}\vert_{\eps=0}\tfrac{d}{d\delta}\vert_{\de=0}E(\pi_N(u+\eps v+ \de w))$ as
\beq\label{eq:second-var}
d^2E(u)(v,w)
=\int_\Si \ \na v\na w -A(u)(\na u,\na u) A(u)(v,w) \dd v_g.
\eeq
\end{lemma}
\begin{proof}
We write $u_\eps:=\pi_N(u+\eps v)=u+\eps v+O(\eps^2)$ and use that the negative $L^2$ gradient of $E$ is given by 
$ \tau_g(u_\eps)=\Delta_g u_\eps+A(u_\eps)(\na u_\eps,\na u_\eps)=P_{u_\eps}(\Delta_g u_\eps)$  to compute
\beqas 
d^2E(u)(v,w)&=-\tfrac{d}{d\eps}\vert_{\eps=0} \int\tau_g(u_\eps) w \, \ed v_g 
=-\tfrac{d}{d\eps}\vert_{\eps=0}\int \Delta_g u_\eps P_{u_\eps}(w)\dd v_g\\
&=\int -\Delta_g v \,w \,dv_g +\int \Delta_g u\, (-dP_u)(v)(w)\dd v_g\\
&=\int\na v\na w \dd v_g+\int \Delta_g u A(u)(v,w) \dd v_g
\eeqas
which gives the claim as the normal component of $\Delta_g u $ is $-A(u)(\na u,\na u)$. 
\end{proof}

We note that  if $u\in W^{1,p}(\Si)$ for some $p>2$, so in particular if $u=z\in \ZZ$, then $d^2E$ has a unique extension to a continuous bilinear form on $\Gamma^{H^1}(u^*TN)$ and we will in the following consider $d^2 E$ on this space. 

The above expression, combined with  
\eqref{est:z-by-weight}, immediately implies that 
$d^2E$ 
is uniformly bounded 
 on the (non-compact) set $\ZZ$ of adapted bubbles equipped with the weighted norms $\norm{\cdot}_z$ in the sense that
\beq
\label{est:bound-sec-var}
\abs{d^2E(z)(w,v)}\leq C\norm{w}_z\norm{v}_z \text{ for every } z\in\ZZ \text{ and all } v,w\in \Gamma^{H^1}(z^*TN).
\eeq
We also note that differences of second variation terms evaluated at different maps $\hat u,\tilde u \in H^2(\Si, N)$ and corresponding tangent vector fields $\tilde v_{1,2}\in \Gamma^{H^1}(\tilde u^*TN)$, $ \hat v_{1,2}\in \Gamma^{H^1}(\hat u^*TN)$ 
can be bounded by 
\beqa \label{est:diff-second-var}
\abs{d^2E(\tilde u)(\tilde v_1,\tilde v_2)-d^2E(\hat u)(\hat v_1,\hat v_2)}&\leqs 
\int \abs{\na (\hat v_1-\tilde v_1)}\abs{\na \hat v_2}+ \abs{\na \tilde v_1}\abs{\na (\hat v_2-\tilde v_2)}\\
 &+\int (\abs{\na  \tilde u}^2+\abs{\na \hat u}^2) [\abs{\hat v_1-\tilde v_1}\abs{\hat v_2}+\abs{\hat v_2-\tilde v_2}\abs{\tilde v_1} ]
 \\
 &+\int \abs{\hat v_1}\abs{\hat v_2} [\abs{\hat u-\tilde u}\abs{\na \hat u}^2+\abs{\na (\hat u-\tilde u)}\abs{\na \hat u}+\abs{\na (\hat u-\tilde u)}^2],
\eeqa 
where all integrals are computed over $(\Si,g)$.

We will use this formula mainly for $\hat u=\zlar\in\ZZ$ and for maps $\tilde u=u_t=\pi_N(z+tw)$, $t\in[0,1]$, that interpolate between $z$ and a map $u=z+w\in H^2(\Si,N)$ with $\norm{w}_{L^\infty}<\de_N$ and for vector fields obtained by projecting suitable
$ v_{1,2}:\Si\to \R^n$ onto the corresponding tangent spaces. In that situation the above formula, combined with \eqref{est:z-by-weight}, gives
\beqa \label{est:diff-second-var-proj}
&\babs{d^2E(u_t)(P_{u_t} v_1, P_{u_t}v_2)-d^2E(z)(P_{z} v_1, P_{z}v_2)} \\
 &\qquad \leqs  \int \abs{w}\abs{\na v_1}\abs{\na v_2}+ \int (\abs{w}\rho_z+\abs{\na w})(\abs{v_1}\abs{\na v_2}+\abs{v_2}\abs{\na v_1})\\
&\qquad \quad + \int \abs{v_1}\abs{v_2}(\abs{w}\rho_z^2+\abs{\na w}\rho_z+\abs{\na w}^2).
\eeqa

\subsection{Uniform definiteness of the second variation orthogonal to $\ZZ$} $ $\\
One of the key features of our set of adapted bubbles is that the second variation of the energy is uniformly definite in directions orthogonal to $\ZZ$. Namely we prove

\begin{lemma} \label{lemma:definite}
Let  $\bbs:S^2\to N$ be any harmonic map and let $\ZZ$ be the set of adapted bubbles defined above. Then there exists
$c_0>0$ 
so that for every $z\in \ZZ$ we can write the orthogonal complement $\VV_z$ of $T_z\ZZ$ in $(\Gamma^{H^1}(z^*TN),\langle\cdot,\cdot\rangle_z)$ as an orthogonal sum $\VV_z=\VV_z^+\oplus \VV_z^-$
of spaces which are so that for $v^\pm\in \VV_z^\pm$
$$d^2E(z)(v_+,v_-)=0 \text{ and } \pm d^2E(z)(v_{\pm},v_{\pm})\geq c_0 \norm{v_\pm}^2_{z}. $$
Here $\langle\cdot, \cdot\rangle_z$ and the associated norm are as in Definition \ref{def:inner-product}.
\end{lemma}
 We remark that the analogue property holds true for the manifold used in \cite{Simon} in the proof of the classical \Lojns-Simon inequality, and there directly follows as maps in this manifold are $C^k$ close to $\bbs$ and as $T_\bbs\HH=\ker(L_\bbs)$. 

Here we have to proceed with more care since our set $\ZZ$ is non-compact.
As in the proof of the corresponding statement  \cite[Lemma 3.6]{MRS} for the $H$-surface energy we prove this result by establishing a uniform gap around $0$ in the spectrum of the projected Jacobi operator, though here use energy considerations rather than Lorentz-space techniques as we do not have the  explicit divergence structure present in \cite{MRS}. 

\begin{proof}[Proof of Lemma \ref{lemma:definite}]
We note that for each fixed $z\in \ZZ$ the space $\Gamma^{H^1}(z^*TN)$ equipped with $\innz$
is a Hilbert-space so Riesz's representation theorem allows us to consider the corresponding
Jacobi-operator $\tilde L_z:\Gamma^{H^1}(z^*TN)\to \Gamma^{H^1}(z^*TN)$ which is
characterised by 
$$d^2E(z)(v,w)=\langle \tilde L_zv,w\rangle_z \text{ for every } v,w\in \Gamma^{H^1}(z^*TN).$$
From the definition of the inner product and Lemma \ref{lemma:2ndvar}
it is easy to see that $\tilde L_z=\Id-K_z$ for $K_z:\Gamma^{H^1}(z^*TN)\to \Gamma^{H^1}(z^*TN)$ characterised by 
\beq
\label{eq:PDE-Kz}
P_z(-\Delta_g K_z(v))+\rho_z^2K_z(v)=b_z(v):=\rho_z^2 v+\sum\langle A(z)(\na z,\na z), \nu^j_z\rangle P_z(d\nu^j_z(v)),\eeq
$\{\nu^j_p\}$ a local orthonormal frame of $T_p^\perp N$. We note that the right hand side is bounded by $\abs{b_z(v)}\leq C\rho_z^2\abs{v}$ so 
as $ \rho_z\in L^\infty(\Si)$ we have that $K_z$ is a compact and selfadjoint operator on
$(\Gamma^{H^1}(z^*TN), \langle \cdot,\cdot\rangle_z)$. 

In order to construct the desired splitting of $\VV_z$ we then consider the projected Jacobi operator $\hat L_z:=P^{\VV_z}\circ \tilde L_z\vert_{\VV_z}=\Id_{\VV_z}-\hat K_z$, 
 $P^{\VV_z}$ the orthogonal projection onto $\VV_z$, 
 which can be equivalently characterised as the unique operator $\hat L_z:\VV_z\to\VV_z$ 
which is so that 
$$d^2E(z)(v,w)=\langle \hat L_zv,w\rangle_z \text{ for every } v,w\in \VV_z.$$
As  $\hat K_z=P^{\VV_z}\circ K_z\vert_{ \VV_z}$
is also selfadjoint and compact 
we know that the eigenvalues of $\hat L_z$ are real and tend to $1$ and that 
there exists an orthonormal eigenbasis of 
  $(\VV_z, \innz)$. 
It hence suffices to show that there exists $c_0>0$ so that, after increasing $\la_1$ and decreasing $\si_1$ if necessary, 
 none of the operators $\hat L_z$, $z\in \ZZ$, has an eigenvalue in $[-c_0,c_0]$. This will imply that the lemma holds true for $\VV_z^\pm$ chosen as span of the eigenfunctions to positive respectively negative eigenvalues. 

To prove this eigenvalue gap we argue by contradiction. Suppose there exist sequences of
adapted bubbles $z_i=z_{\la_i}^{a_i,\om_i}\in \ZZ$ with $\la_i\to \infty$ and $\om_i\to \bbs$ and 
of elements $v_i\in \VV_{z_i}$,  normalised to $\norm{v_i}_{z_i}=1$, so that 
$\hat L_{z_i} v_i=\mu_i v_i$ for some  $\mu_i\to 0$. 
We first claim that there exists a number $c_1>0$ so that for all sufficiently large $i$ 
\beq
\label{est:ev-energy-smallball}
\int_{\DD_{\la_i^{-1/3}}} \rho_{\la_i}^2  \abs{\tilde v_i}^2\dd x\geq c_1, \quad \text{ where }\tilde v_i:=v_i\circ F_{a_i}^{-1}.
\eeq
To see this we use that $\rho_{z_i}$ is of order $O(\la_i^{-1})$ away from the ball $B_{\iota}(a_i)$, while
 $\rho_{\la_i}=\frac{\la_i}{1+\la_i^2\abs{x}^2}\leq \la_i^{-\frac13}$ on $\DD_{r_0}\setminus \DD_{\la_i^{-1/3}}$.
We can thus bound 
\beqas 
\mu_i&=\langle \hat L_{z_i}v_i,v_i\rangle_{z_i}=\norm{v_i}_{z_i}^2-\langle K_{z_i}(v_i),v_i\rangle_{z_i} =1-\int_\Si b_{z_i}(v_i)v_i \dd v_g\\
&\geq 1-C\int_{\DD_{\la_i^{-1/3}}} \rho_{\la_i}^2\abs{\tilde v_i}^2 \dd x
-C\la_i^{-\frac23}\norm{v_i}_{L^2(\Si)}^2\geq 1-C\la_i^{-\frac23}\log(\la_i)- C\int_{\DD_{\la_i^{-1/3}}} \rho_{\la_i}^2\abs{\tilde v_i}^2 \dd x,
\eeqas
where we use that the Poincar\'e hyperbolic metric is uniformly equivalent  to the Euclidean metric on $\DD_{r_0}$ in the penultimate step and \eqref{est:Lp-by-z-norm} in the last step. 
As $\mu_i\to 0$ this yields the claimed lower bound \eqref{est:ev-energy-smallball} for all sufficiently large  $i$.

We now proceed to construct a sequence  of maps $w_i: S^2\to \R^n$ that converges  to a limit $w_\infty$ which 
is a non-trivial Jacobifield at $\bbs$ but also orthogonal 
to 
\beq
\label{def:Xbbs}
X_\bbs:= T_\bbs \{\bb\circ M: \bb \in \HHz(\bbs), M\in \Mob \}.\eeq
This leads to the desired contradiction since $\HHz(\bbs)$ is chosen in a way that ensures that $X_\bbs$ agrees with the space $\ker(L_\bbs)=\ker(\hat L_\bbs)$
 of Jacobifields at $\bbs$, compare Lemma \ref{lemma:Simon}.

To construct these maps $w_i$ 
we first define  
$$\tilde w_i=\phi_i \tilde v_i(\la_i^{-1}\cdot )+(1-\phi_i) \bar v_{i}:\R^2\to \R^n,$$
where we let $\bar v_{i}=
 \fint_{A_i}\tilde v_i dx$ be the meanvalue over the annulus 
$A_i= \DD_{2\la_i^{-1/3}}\setminus \DD_{\la_i^{-1/3}}$ and set 
$\phi_i(x)=\phi(\la_i^{-\frac23}\abs{x})$ for some fixed $\phi\in C_c^{\infty}([0,2),[0,1])$ with $\phi\equiv 1$ on $[0,1]$.

As $\tilde w_i$ is constant near infinity, we have $w_i:=\tilde w_i\circ \pi^{-1}\in H^1(S^2,\R^n)$ and we can bound 
\beqas
\norm{w_i}_{H^1(S^2)}^2&=\int_{\R^2}\abs{\na \tilde w_i}^2+\abs{\partial_{x_1}\pi\wedge \partial_{x_2}\pi}\abs{\tilde w_i}^2 \dd x=
\int_{\R^2}\abs{\na \tilde w_i}^2+\thalf\abs{\na \pi}^2\abs{\tilde w_i}^2 \dd x\\
&
 \leqs \int_{\DD_{2\la_i^{-1/3}}}\abs{\na \tilde v_i}^2+\abs{\na \pi_{\la_i}}^2 \abs{\tilde v_i}^2
+\fint_{A_i}\abs{\tilde v_i-\bar{v_i}}^2 
+\abs{\bar v_{i}}^2
\norm{\na \pi}_{L^2(\R^2\setminus \DD_{\la_i^{2/3}})}^2 
\\
&\leqs \norm{v_i}_{z_i}^2+C\la_i^{-4/3}\abs{\bar v_{i}}^2\leqs  \norm{v_i}_{z_i}^2
\eeqas
where the last step follows as
$\rho_{\la_i}\geq 
 c \la_i^{-\frac13}$ on $A_i$ and thus
 $\la_i^{-\frac43}\abs{\bar v_{i}}^2\leqs \la_i^{-\frac23}\int_{A_i}\abs{\tilde v_i}^2 \dd x\leqs \norm{v_i}_{z_i}^2$.
 
After passing to a subsequence the maps $w_i$ thus converges to a limit $w_\infty$ weakly in $H^1(S^2,\R^n)$, strongly in $L^2(S^2,\R^n)$ and almost everywhere. 
Away from the shrinking discs $\pi(\R^2\setminus \DD_{\la_i^{2/3}})\subset S^2$ the maps
$w_i=v_i \circ (\pi_{\la_i}\circ F_{a_i})^{-1}$ are tangential to $N$ along 
\beq
\label{eq:hatzi}
\hat z_i:=z_i\circ (\pi_{\la_i}\circ F_{a_i})^{-1}=\omega_{i}+O(\la_i^{-1})\to \bbs
\eeq
so the limit $w_\infty$ must be tangential along $\bbs$, i.e. an element of $\Gamma^{H^1}(\bbs^*TN)$.
Furthermore, $w_\infty$ is non-trivial as  $w_i\to w_\infty$ strongly in $L^2(S^2,\R^n)$ and as the $L^2$ norms of the maps $w_i$ are bounded away from zero thanks to 
\eqref{est:ev-energy-smallball}.
We now want to prove that $w_\infty$ is orthogonal to the space $X_\bbs$ with respect to 
the inner product 
\beq
\label{def:inner-prod-S2}
\langle v,w\rangle:=\int_{S^2} \na w\na w+c_\gamma w v dv_{g_{S^2}}.
\eeq
\begin{rmk}
\label{rmk:inner-prod-S2}
Here we set  
$c_\gamma=\frac14$ if $\gamma=1$ as this ensures that 
$(\pi_\la\circ F_a)^*g_{S^2}=c_\gamma \rho_{\zlar}^2g $ on $B_\iota(a)$
while we set 
 $c_\gamma=1$ if $\gamma\geq 2$ and use that in this case 
$((\pi_\la\circ F_a)^*g_{S^2})(p)=c_\gamma \rho_{\zlar}^2(1+O(\dist_g(p,a)^2)g(p)$ for $p\in B_\iota(a)$.
\end{rmk}
We can use 
 the following lemma, see Appendix \ref{appendix:technical} for a sketch of the proof. 
\begin{lemma}\label{lemma:proj-on}
Let $z_i=z_{\la_i}^{a_i,\bb_i}$ be a sequence of adapted bubbles for which $\la_i\to \infty$ and $\bb_i\to \bbs$. Then there exist bases $\{e_j^{i}\}_{j=1}^K$ of $T_{z_i}\ZZ$ that are orthonormal with respect to $\langle\cdot,\cdot\rangle_{z_i}$ and that converge to an orthonormal basis $\{e_j^{\infty}\}_{j=1}^K$ of 
$ (X_{\bbs}, \langle\cdot,\cdot\rangle)$
in the sense that 
$\hat e_j^{i}:= e_j^{i}\circ (\pi_{\la_i}\circ F_{a_i})^{-1}\to e_j^\infty \text{ smoothly locally on } S^2\setminus\{p^*\}$
while 
\beq
\label{claim:basis2}
\lim_{\Lambda\to \infty}\limsup_{i\to \infty} \|e_j^{i}\|_{z_i,\Sigma\setminus B_{\Lambda\la_i^{-1}}(a_i)}=0 \text{ for } j=1,\ldots,K.\eeq
\end{lemma}

As $w_i\weakto w_\infty$ in $H^1(S^2)$
we obtain that for $j=1,\ldots,K$, 
\beqas
\langle w_\infty,e_j^\infty\rangle&= \lim_{\La\to \infty} \int_{\pi(\DD_\La)} \na w_\infty\na e_j^\infty+c_\gamma w_\infty e_j^\infty \dd v_{g_{S^2}}\\
&=\lim_{\La\to\infty}\lim_{i\to \infty}\int_{\pi(\DD_{\La})} \na w_i\na \hat e_j^i+c_\gamma w_i \hat e_j^i \dd v_{g_{S^2}} \\
&=\lim_{\La\to\infty}\lim_{i\to \infty}\int_{F_{a_i}^{-1}(\DD_{\La\la_i^{-1}})}
\nabla v_i\nabla e_j^{i} + v_i e_j^{i}\rho_{z_i}^2 \dd v_g\\
&=-\lim_{\La\to\infty}\lim_{i\to \infty}\int_{\Si\setminus F_{a_i}^{-1}(\DD_{\La\la_i^{-1}})} \nabla v_i\nabla e_j^{i}+v_i e_j^{i}\rho_{z_i}^2\dd v_g=0,
\eeqas
where we use Remark \ref{rmk:inner-prod-S2} in the third step, the orthogonality of $v_i$ to $T_{z_i}\ZZ$ 
in the penultimate step and \eqref{claim:basis2} as well as that $\norm{v_i}_{z_i}=1$ in the last step.

Having thus shown that $w_\infty\perp X(\bbs)=\ker(L_\bbs)$
it now remains to show that
\beq \label{est:claim-Jacobi}
d^2E(\bbs)(w_\infty,\eta)=0 \text{ for all }\eta\in \Gamma^{H^1}(\bbs^*TN).\eeq 
We note that this is trivially true if $\eta$ itself is a Jacobi-field and that it hence suffices to consider $\eta\in \Gamma^{H^1}(\bbs^*TN)$ with $\eta\perp  \ker(\hat L_\bbs)$.

Given such an $\eta$ 
we set $\bar\eta_i:=\fint_{A_{r_0}}\eta\circ \pi_{\la_i} dx$, $A_{r_0}:=\DD_{r_0}\setminus \DD_{\frac{r_0}{2}}$, and define $\eta_i\in \Gamma^{H^1}(z_i^*TN)$ as $\eta_i=P_{z_i}(\bar\eta_i)$ on $\Si\setminus B_{\iota}(a_i)$
while for $p=F_{a_i}^{-1}(x)\in B_{\iota}(a_i)$ 
$$\eta_i(p):=P_{z_i(p)}(\psi(x)\eta(\pi_{\la_i}(x))+(1-\psi(x))\bar\eta_i)$$
for a fixed cut-off $\psi\in C_{c}^\infty(\DD_{r_0})$ with $\psi\equiv 1$ on $\DD_{\frac{r_0}{2}}$.

As $d^2E$ is conformally invariant and as $\la_i^{-1/3}\leq\half r_0$ for sufficiently large $i$ 
we get  
\beq
\label{est:123}
d^2E(z_i)(v_i,\eta_i)=\int_{\pi(\DD_{\la_i^{2/3}})}\na (P_{\hat z_i}\eta)\na w_i+A(\hat z_i)(\na \hat z_i,\na \hat z_i)A(\hat z_i)(P_{\hat z_i}\eta,w_i) dv_{g_{S^2}} +\text{err}_i,\eeq
for $\hat z_i$ defined by \eqref{eq:hatzi} and $\abs{\text{err}_i}\leq C\norm{v_i}_{z_i}\norm{\eta_i}_{z_i,\Si\setminus F_{a_i}^{-1}(\DD_{\la_i^{-1/3}})}$.
We recall that 
 $\rho_{z_i}\leqs \la_i^{-1}$ away from $B_{\tilde \iota}(a)$  and 
note 
that a short calculation, similar to the proof of \eqref{est:MV} carried out in the appendix, gives $\abs{\bar\eta_i}\leqs (\log\la_i)^\half \norm{\eta}_{H^1(S^2)}$. As $\norm{v_i}_{z_i}=1$, this allows us to conclude that
\beqas
\abs{\text{err}_i}
&
\leqs \la_i^{-1} \abs{\bar \eta_i}  
 +\norm{\eta\circ\pi_{\la_i}-\bar\eta_i}_{L^2(A_{r_0})}
+\norm{\rho_{\la_i}\abs{\eta\circ\pi_{\la_i}}+\abs{\na (\eta\circ\pi_{\la_i})}}_{L^2(\DD_{r_0}\setminus \DD_{\la_i^{-1/3}})}\\
&\leqs \la_i^{-1}(\log\la_i)^\half \norm{\eta}_{H^1(S^2)} +\norm{\eta}_{H^1(\pi(\R^2\setminus \DD_{\la_i^{2/3}}))}\to 0.
\eeqas
Combined with 
 $w_i\weakto w_\infty$ in $H^1(S^2)$ and
$\norm{\hat z_i-\bbs}_{C^1(\pi(\DD_{\la_ir_0}))}\to 0$, which follows as $\bb_i\to \bbs$ smoothly on $S^2$ and 
$\norm{\hat z_i-\om_{i}}_{C^1(\pi(\DD_{\la_ir_0}))}\leq C\la_i^{-1}\to 0$, this shows that the right hand side of \eqref{est:123} converge to 
$d^2E(\bbs)(w_\infty,\eta)$.
We thus conclude that
\beqas
\abs{d^2E(\bbs)(w_\infty,\eta)}&=\lim_{i\to \infty}\abs{d^2E(z_i)(v_i,\eta_i)}\leq \abs{\langle \hat L_{z_i}v_i, P^{\VV_i}\eta_i\rangle}+C\norm{v_i}_{z_i}\norm{P^{T_{z_i}\ZZ}\eta_i}_{z_i}\\
&\leq \abs{\mu_i}\norm{\eta_i}_{z_i}+C\norm{P^{T_{z_i}\ZZ}\eta_i}_{z_i}.
\eeqas 
As $\norm{\eta_i}_{z_i}\leq C\norm{\eta}_{H^1(S^2)}+C\la_i^{-1}\abs{\bar\eta_i} $ is uniformly bounded and as we have assumed that $\mu_i\to 0$, we know that the first term in this estimate tends to zero as $i\to \infty$. We can furthermore use Lemma \ref{lemma:proj-on} to see that  for $j=1,\ldots,K$
\beqas
\lim_{i\to\infty}\langle e_i^j, \eta_i\rangle_{z_i}
&=\lim_{\Lambda\to \infty}\lim_{i\to\infty} \langle e_i^j,\eta_i\rangle_{z_i,\Si\setminus F_{a_i}^{-1}(\DD_{\Lambda \la_i^{-1}})}+\lim_{\Lambda\to \infty}\lim_{i\to\infty} \langle \hat e_i^j,P_{\hat z_i}(\eta)\rangle_{\pi(\DD_\Lambda)}\\
&= 0+\lim_{\Lambda \to \infty}\langle e_\infty^j,\eta\rangle_{\pi(\DD_\Lambda)}=\langle e_\infty^j,\eta\rangle =0,
\eeqas
where the last step follows as $\eta\perp \ker(L_\bbs)=X_\bbs$. Hence 
 $\norm{P^{T_{z_i}\ZZ}\eta_i}_{z_i}\to 0$ and we indeed obtain that $d^2E(\bbs)(w_\infty, \eta)=0$. Thus $w_\infty\in \ker(L_\bbs)$ 
  contradicting the previously established fact that 
$w_\infty$ is a non-trivial element of $(\ker(L_\bbs))^\perp$. 
\end{proof}

\subsection{Expansion of the energy on the set $\ZZ$ of adapted bubbles} $ $\\
The goal of this section is to identify variations in the space of adapted bubbles for which the 
leading order term in the energy expansion appears with a known sign and scaling. 

In the integrable case, where all elements of $\ZZ$ are built out of harmonic maps $\om:S^2\to N$, we will only need to consider variations $(z_\eps)$ induced by a change of the bubble parameter. In the general case we will additionally need to consider $(z_\eps)$  induced by variations of the underlying maps $\om^{(\eps)}\in \HHone$. 
To treat both types of variations at the same time we first show. 
\begin{lemma} \label{lemma:main-term}
For any variation $z_\ep=z_{\la_\eps}^{a, \om^{(\eps)}}$ in $\ZZ$ for which \eqref{ass:var} holds we have
\beq
\label{eq:expansion-E-general}
\ddepsz E(z_\eps)=\tfrac{d\la}{d\eps}\cdot 
\int_{\DD_{\frac{r_0}{2}}} j_\la^{a,\om} \Delta \partial_\la \omlrt \dd x+\ddepsz E(\om^{(\eps)}) + \text{err} 
\eeq
for an error term that is bounded by 
\beq\label{est:exp-general-error}
\abs{\text{err}}\leq C \la^{-3} +
C\la^{-2} 
\big[\norm{\peps \om^{(\eps)}}_{C^2(S^2)}+\norm{\tauS(\om)}_{C^1(S^2)}\big].
\eeq

\end{lemma}

\begin{proof}[Proof of Lemma \ref{lemma:main-term}]
Let $z =z_{\la}^{a,\bb}
\in \ZZ$ and let $\zeps=z_{\leps}^{a,\bbeps}$ be a variation for which \eqref{ass:var} holds. To lighten the notation  we write for short 
$v=v_{\la}^{a,\bb}$, $j=j_{\la}^{a,\bb}$
 and denote the corresponding variations by 
$\peps v:=\ddepsz v_{\leps}^{a,\bbeps}$, $\peps\jeps:=\ddepsz j_{\leps}^{a,\bbeps}$ and $\peps \tilde\om_\la= \ddepsz (\omega^{(\eps)}\circ \pi_{\leps})$. 

We first remark that away from the ball $B_\iota(a)$ we have $\Delta_g v=0$ 
as the derivatives of the Green's function are harmonic functions. 
Combined with the estimate \eqref{est:elart}  on the error term in \eqref{eq:writing-v-lart} and with  $\tau_g(z)=P_z (\Delta_g z)=P_z(d\pi_N(v)(\Delta_g v)+d^2\pi_N(v)(\na v,\na v))$ 
we hence get that 
\beq \label{est:v-Delta-away}
\abs{\Delta_g v}+\abs{\peps \Delta_g v}+\abs{\na v}^2+ \abs{\peps \na v}^2
+\abs{\tau_g(z)}+\abs{\peps \tau_g(z)}
\leqs\la^{-2}  
\text{ on }\Sigma\setminus B_{\tilde \iota}(a).
\eeq

We also note that \eqref{est:v-away} yields a bound of
$\abs{\peps z} \leqs \la^{-1} +\etaz$ on this set, 
where here and in the following we write for short $\etaz:= \abs{\peps \bbeps(p^*)}$.
We thus obtain that  
\beqa \label{est:est1}
\ddepsz E(z)&=-\int_\Si \peps z \cdot \tau_g(z) \dd v_g 
=-\int_{B_{\tilde \iota}(a)}\peps z \cdot \Delta_g z \dd v_g+O(\la^{-3}+\etaz\la^{-2} )\\
&=-\int_{\DD_{\frac{r_0}{2}}} \peps \tilde z \cdot \Delta \tilde z dx +O(\la^{-3}+\etaz\la^{-2} )
\eeqa
for $\tilde z=z\circ F_a^{-1}$. Here and in the following we can 
carry out all computations on $\DD_{\frac{r_0}{2}}$ with respect to the Euclidean metric as the above integral is conformally invariant. 
 On this set we can write $\tilde z=\pi_N(\omlrt+j)$ as 
\beq \label{eq:repr-zla-error}
\tilde z =\omlrt+P_{\omlrt}(j)+E= \omlrt+j-P^\perp_{\omlrt}(j)+E
\eeq
where the lower order error term 
\beqa
E
&=\int_0^1\tfrac{d}{dt}\pi_N(\omlrt+tj) \dd t-P_\omlrt(j)=\int_0^1 (d\pi_N(\omlrt+tj)-d\pi_N(\omlrt))(j)\dd t
\eeqa
satisfies the estimates 
\beqa
\label{est:Ela-new}
\abs{E}+ \abs{\peps E}\leqs \la^{-2}\abs{x}^2, \, \abs{\na E}+\abs{\peps \na E}\leqs \la^{-2}\abs{x}, \, \abs{\Delta E}+\abs{\peps \Delta E}\leqs \la^{-2}  .
\eeqa
Here and in the following we use that 
\beq
\label{est:various-bounds}
\abs{j}+\abs{\peps  j}\leqs \la^{-1}\abs{x}, \quad \abs{\peps \na j}\leqs  \la^{-1} , \quad \abs{\peps \na \pi_\la}\leqs \rho_\la \text{ and } \rho_\la\abs{x}\leq 1
\eeq 
while
\beq
\label{est:var-om-eta} 
\abs{\partial_\eps\omla}\leqs \frac{(\abs{\peps \la}+\la)\abs{x}}{1+\la^2\abs{x}^2}+\eta_0 \leqs (1+\la\abs{x})^{-1}+\eta_0 \text{ and } \abs{\peps\na\omla}\leqs \rho_z.
\eeq
In the following it will also be useful to note that this implies that
\beq
\label{est:pepsz}
\abs{\partial_\eps \tilde z}\leqs (1+\la\abs{x})^{-1}+\eta_0, 
\eeq
that we can trivially bound 
\beq
\label{est:Delta-om-trivial}
 \abs{\Delta \omla}+ 
\abs{\peps \Delta \omla}
\leqs
\big[\norm{\om}_{C^2(S^2)}+\norm{\peps \bbeps}_{C^2(S^2)}\big]
\big[\abs{\na \pi_\la}^2 +\abs{\peps \la}\abs{\na\pi_\la}\abs{\partial_\la \na \pi_\la}\big]  \leqs \rho_z^2,
  \eeq
  and that we have estimates of   
  \beq
  \label{est:norms-z}
\norm{\rho_\la\abs{x}+(1+\la\abs{x})^{-1}
}_{L^2(\DDr)}\leqs \la^{-1}(\log\la)^\half, \quad 
\norm{\rho_\la\abs{x}+(1+\la\abs{x})^{-1}}_{L^1(\DDr)}
\leqs \la^{-1} .
\eeq 
From \eqref{est:est1}, \eqref{eq:repr-zla-error} and $\Delta_g j=0$ we thus obtain that 
\beqa
\label{est:est1-2}
\ddepsz E(z)&=
-\int_{\DD_{\frac{r_0}{2}}}
\partial_\eps \omlrt\Delta(\omlrt-P_{\omlrt}^\perp j)+\peps(P_{\omlrt}j)
\Delta\omlrt +\text{err}_1+O(\la^{-3}+\etaz\la^{-2})
\eeqa
for an error term that is bounded by 
\beqas
\abs{\text{err}_1}&\leqs 
\int_\DDr\abs{\peps(P_{\omlrt}j)}
\abs{\Delta  P_{\omlrt}^\perp j} +\abs{\partial_\eps E}(\abs{\Delta  \omlrt}+\abs{\Delta  P_{\omlrt}^\perp j})+\abs{\Delta  E}\abs{\partial_\eps \zeps}\\
&\leqs \la^{-2} \int_ {\DD_{\frac{r_0}{2}}}\rho_\la^2 \abs{x}^2+\rho_\la \abs{x}+ \la^{-2} \etaz +\la^{-2} \norm{( 1+\la\abs{x})^{-1}}_{L^1(\DDr)} \leqs  \la^{-3}+ \eta_0\la^{-2}.
\eeqas
We then note that  
\beqa\label{est:rewrite-1-2}
-\int_{\DD_{\frac{r_0}{2}}}\peps \omla\Delta  \omla= 
-\int_{\R^2}\peps \omla\tau (\omla)+\text{err}_2=
\ddeps E(\bbeps)+\text{err}_2
\eeqa
where $\text{err}_2= \int_{\R^2\setminus \DD_{\frac{r_0}{2}}}\peps \omla  \tau (\omla)$ is also bounded by
\beqa
\abs{\text{err}_{2}}
\leqs \norm{\peps \omla}_{L^\infty(\R^2\setminus \DD_{\frac{r_0}{2}} )}
\norm{\tauS(\om)}_{L^\infty(S^2)} \int_{\R^2\setminus \DD_{r_0/2}}\abs{\na \pi_\la}^2 =O(\la^{-3}+\etaz \la^{-2}).
\eeqa 
Here and in the following we use that $\pi_\la$ is conformal and hence  
\beq
\label{eq:trafo-tension-la}
\tau (\omla)=\thalf \abs{\na \pi_\la} ^2\cdot \tau_{g_{S^2}}(\om) \circ \pi_\la.
\eeq
Rewriting $$\peps(P_{\omlrt}j)\cdot \Delta \omla=\peps(P_{\omla}j\cdot \Delta  \omla)-P_{\omlrt}j \cdot \peps\Delta \omla =\peps(j\cdot\tau (\omla))-P_{\omlrt}j \cdot \peps\Delta \omla$$ 
and setting  $\text{err}_3:=-\int_{\DD_{\frac{r_0}{2}}}  \peps(j\tau (\omla))$
we hence obtain from \eqref{est:est1-2} and \eqref{est:rewrite-1-2} that 
\beqas
\ddepsz E(\zeps)&=\ddepsz E(\bbeps)
+\int_{\DD_{\frac{r_0}{2}}}P_{\omlrt}j\cdot \peps\Delta \omla
+\Delta (P_{\omlrt}^\perp j)\cdot \partial_\eps\omla +\text{err}_3
+O(\la^{-3}+\etaz\la^{-2})\\
&=\ddepsz E(\bbeps)
+\int_{\DD_{\frac{r_0}{2}}}j \cdot \peps\Delta \omla
+\text{err}_3+\text{err}_4+
O(\la^{-3}+\etaz\la^{-2})
\eeqas
where we integrate by parts in the second step. We can use \eqref{est:var-om-eta} as well as that 
$j\in T_{\om(p^*)}N$
to estimate the resulting the boundary term by 
\beqas 
\abs{\text{err}_4}&\leq 
\norm{\partial_\eps \omla}_{C^1(\partial \DD_{\frac{r_0}{2}})}
\norm{(P^{\perp}_{\omlrt}-P^\perp_{\om(p^*)})(j)}_{C^1(\partial \DD_{\frac{r_0}{2}})}
=O(\la^{-3}+\etaz\la^{-2}),
\eeqas
while combining \eqref{eq:trafo-tension-la} with 
\eqref{est:various-bounds} and \eqref{est:norms-z} allows us to estimate
\beqa
\abs{\text{err}_3}&\leqs \la^{-1} \big[\norm{\tauS(\om)}_{C^1(S^2)}+ \norm{\peps \bbeps}_{C^2(S^2)} \big] \norm{\abs{x} \rho_z^2}_{L^1(\DDr)}\\
&\leqs \la^{-2}  (\norm{\tauS(\om)}_{C^1(S^2)}+ \norm{\peps \bbeps}_{C^2(S^2)}).
\eeqa
We finally remark that 
$\int_{\DD_{\frac{r_0}{2}}}j \peps\Delta \omla=\frac{d\la}{d\eps} \int_{\DD_{\frac{r_0}{2}}} j \partial_\la\Delta \omla+\text{err}_5$
for 
$$\abs{\text{err}_5}\leq \int_{\DD_{\frac{r_0}{2}}}\abs{j} \abs{\Delta ( (\peps \bbeps)\circ \pi_\la)} \leqs \la^{-1} \norm{\peps\bbeps}_{C^2} \norm{\abs{x} \rho_z^2}_{L^1(\DDr)}\leqs \la^{-2} \norm{\peps\bbeps}_{C^2}.$$
Altogether this yields the claim of Lemma \ref{lemma:main-term}.
\end{proof}

We now show that the integral $\int j\partial_\la \Delta \om_\la$ appearing in \eqref{eq:expansion-E-general} has a given sign and scaling in $\la$ and indeed essentially only depends on $a\in \Si$, $\la$ and $\abs{d\om(p^*)}$.

To state this in detail 
we first note that as
$d\bbs(p^*)\neq 0$ we 
can always assume that 
$\si_1>0$ is chosen small enough to ensure that
\beq
\label{ass:alpha}
\abs{d \bb(p^*)}\geq \half \abs{d \bbs(p^*)}>0\text{ for all } \bb\in \HH_1^{\si_1}(\bbs).\eeq
Writing for short
$\al_\bb:=\frac{1}{\sqrt{2}}\abs{d \bb (p^*)}_{g_{S^2}} $
we then note that 
if $\om$ is harmonic then
the vectors $\{\al_\bb^{-1}\na_{e_1} \bb(p^*),\al_\bb^{-1} \na_{e_2} \bb(p^*)\}$ are orthonormal since harmonic maps from $S^2$ are weakly conformal. 
 While this is not true for 
general elements $\om \in \HHone$  in the non-integrable case, the above will still hold up to a small error as elements of $\HHone$ are $C^k$ close to the harmonic map $\bbs$. 

So given any number $\etanew>0$ we can assume that $\si_1>0$ is chosen small enough so that 
for any $\bb\in \HHone$
there exists a matrix $S_\bb\in M(n)$ with
\beq
\label{def:Som}
d\bb(p^*)(e_i)=\al_\bb S_\bb e_i\in \R^n, \quad i=1,2, \quad \text{ and } \abs{S_\bb^TS_\bb-\Id}\leq \etanew.
\eeq
Here we denote by $\{e_i\}$ both the standard basis of $\R^3$ and of $\R^n$ as appropriate. We note that as $S_\bb$ will only be applied to elements of $\R^2\times \{0\}\subset \R^n$ in the construction below the particular choice of $S_\bb$ in the other directions is irrelevant. 
We can then prove

\begin{lemma}
\label{lemma:main-term-scaling}
For any $\om\in \HHone$, $\la\geq \la_1$ and $a\in \Si$ we have 
$$
\int_{\DD_{\frac{r_0}{2}}} j_\la^{a,\om} \Delta_g \partial_\la \omlrt \dd v_g
=4\pi \abs{d\om(p^*)}^2\mathcal{J}(a)\la^{-3}+O(\la^{-4})+O(\abs{S_\om^TS_\om-\Id}\la^{-3})$$
for $S_\om$ as in \eqref{def:Som} and 
$
\mathcal{J}(a):=\lim_{x\to 0} (\partial_{y_1}\partial_{x_1}+\partial_{y_2}\partial_{x_2}) G_a(x,0),$ where 
 $G_a(x,y)=G(F_a^{-1}(x),F_a^{-1}(y))$ is the function that represents the Green's function in the coordinates $F_a$ introduced in Remark \ref{rmk:Fa}.
\end{lemma}

\begin{rmk}
\label{rmk:J_neg}
We recall from \cite{MRS} that the function $\mathcal{J}$, which depends only on the domain surface $(\Si,g)$, is strictly negative on any surface of positive genus. 
\end{rmk}

\begin{proof}[Proof of Lemma \ref{lemma:main-term-scaling}]
Extending $\om:S^2\to N\hookrightarrow \R^n$ to a neighbourhood of $S^2$ by setting $\om(x)=\om(\abs{x}^{-1}x)$ we can view $d\om(p^*)$ as a map from $\R^3$ to $\R^n$ with $d\om(p^*)(e_3)=0$. 
Thus \eqref{def:Som} allows us to write  
$d\om(p^*)(\partial_\la \pi_\la)= \al_\bb S_\bb (\partial_\la \bar \pi_\la, 0_{n-2})^T,$
where $\bar \pi_\la=(\pi_\la^1,\pi_\lambda^2):\R^2\to \R^2$ is given by the first two components of the rescaled inverse stereographic projection $\pi_\la$. 
We can use this to estimate 
\beqas 
\norm{\partial_\la(\omlrt-\al_\bb S_\bb(\bar \pi_\la,0)^T)}_{C^1(\partial \DD_{\frac{r_0}{2}})}
&
= \norm{[d\om(\pi_\la)-d\om(p^*)](\partial_\la \pi_\la)} 
_{C^1(\partial \DD_{\frac{r_0}{2}})}\\
&\leqs \norm{\om}_{C^2(S^2)} \norm{\pi_\la-p^*}_{C^1(\partial \DD_{\frac{r_0}{2}})}\norm{\partial_\la \pi_\la}_{C^1(\partial \DD_{\frac{r_0}{2}})}\leqs \la^{-3}.
\eeqas 
Integration by parts, using also 
$\norm{j_\la^{a,\om}}_{C^1(\DD_{\frac{r_0}{2}})}=O(\la^{-1})$ 
and $\Delta_g j_\la^{a,\om}=0$, thus gives
\beqas
\int_{\DD_{\frac{r_0}{2}}} j_\la^{a,\om}\Delta \partial_\la \omlrt 
&=\int_{\DD_{\frac{r_0}{2}}}j_\la^{a,\om}\Delta(\al_\bb S_\bb\partial_\la(\bar \pi_\la,0)^T) 
+O(\la^{-4})\\
&=\al_\bb \int_{\DD_{\frac{r_0}{2}}} (S_\om^T j_\la^{a,\om})\cdot (\Delta \partial_\la\bar \pi_\la,0)^T
+O(\la^{-4}).
\eeqas
Writing for short $\hat j_a:=2 \na_y J_a(\cdot,0)-2 \na_y J_a(0,0):\DD_r\to \R^2$
we now recall that 
$j_\la^{a,\om} $ is given by  $j_\la^{a,\om} = \la^{-1} d\om(p^*) (\hat j_a,0)^T=
\la^{-1}\al_\om S_\om (\hat j_a,0)^T$. 
As 
$\la^{-1}\int\abs{\hat j_a}\abs{\Delta \partial_\la \pi_\la}\leqs \la^{-2}\int\abs{x}\rho_\la^2\leqs \la^{-3}$ we thus obtain that 
\beqa \label{est:12345}
\int_{\DD_{\frac{r_0}{2}}} j_\la^{a,\om}\Delta \partial_\la \omlrt \dd x
&=\al_\om^2\la^{-1} \int_{\DD_{\frac{r_0}{2}}}  \hat j_a \Delta \partial_\la\bar \pi_\la \dd x
+O(\la^{-4}+\abs{S_\om^TS_\om-\Id} \la^{-3}).\\
\eeqa
Up to the factor $\al_\bb^2$ and the constant shift in $\hat j_a$,  the leading order term in \eqref{est:12345} is exactly the same as the leading order term obtained in the proof of Lemma 3.7 of \cite{MRS}. Following the argument there we can thus combine the Taylor expansion of $\hat j_a$ with the symmetries of $\Delta \bar \pi_\la=-\abs{\na \pi_\la}^2 \bar \pi_\la=\frac{-16\la x}{(1+\la^2\abs{x}^2)^3}$ 
to compute
\beqas
\al_\om^2\la^{-1} \int_{\DD_{\frac{r_0}{2}}}  \hat j_a \Delta \partial_\la\bar \pi_\la 
&=-\al_\bb^2\la^{-1} \sum_{i=1,2}2 \partial_{x_i}\partial_{y_i}J_a(0)\int_{\DD_{\frac{r_0}{2}}} x_i \partial_\la 
(\tfrac{16\la x_i}{(1+\la^2\abs{x}^2)^3}) +O(\la^{-4})
\\
&=\al_\bb^2\tfrac{8\pi}{\la^3} (\partial_{x_1}\partial_{y_1}J_a(0)+\partial_{x_2}\partial_{y_2}J_a(0))+O(\la^{-4})\\
&
=\al_\bb^2\tfrac{8\pi}{\la^3} \mathcal J(a)+O(\la^{-4}).
\eeqas
Inserting this into \eqref{est:12345} and using that $\al_\om^2=\half \abs{d\om(p^*)}^2$ gives the claim of the lemma. 
\end{proof}

We first use the above lemmas to control the variation of the energy induced by a change 
of the bubble parameter. To this end we note 
that given any $\eta_2>0$ 
we 
can choose $\si_1>0$ sufficiently small to ensure that 
  \beq
\label{ass:tens}
\norm{\tauS(\om)}_{C^1(S^2)}\leq \eta_2 \text{ for all } \om\in \HHone
\eeq
since 
$\bbs$ is harmonic.
For suitable choices of $\si_1$ and $\la_1$ we can thus combine 
Lemmas \ref{lemma:main-term} and \ref{lemma:main-term-scaling} with Remark \ref{rmk:J_neg} to obtain
\begin{cor}
\label{cor:Expansion-scaling}
There exist constants $c_1>0$ and $C<\infty$ so that for any $z=z_{\la}^{a,\om}\in \ZZ$ 
\beq
\label{claim:dE-lamba}
C\la^{-2}\geq -\la \frac{d}{d\la} E(z_{\la}^{a,\om})\geq c_1\la^{-2}.
\eeq
\end{cor}

\begin{rmk} \label{rmk:delta-energy-bubbles}
As an immediate consequence we obtain that 
\beq
\label{est:delta-energy-bubble}
\abs{E(z,\Si)-E(\omega,S^2)}\leq C\la(z)^{-2} \text{ for any } z\in\ZZ.
\eeq
\end{rmk}

In the non-integrable case we furthermore need to control the tension of the underlying map $\om$ if $\om\in\HHone$ is not harmonic. To this end we let $\om_0\in\HH_0^{\si_1}(\bbs)$ and $R\in SO(3)$ be so that $\om=\om_0\circ R$ and set $\om^{(\eps)}=\om_0^{(\eps)}\circ R$ for $\om_0^{(\eps)}$ as in Lemma \ref{lemma:Simon}. 
As such a variation satisfies \eqref{ass:var} and as \eqref{claim:tang-part-tension-below} ensures that 
$$\ddepsz E(\om^{(\eps)})=\ddepsz E(\om_0^{(\eps)})\geq \norm{\tau_{g_{S^2}}(\om_0)}_{L^2(S^2)}= 
\norm{\tau_{g_{S^2}}(\om)}_{L^2(S^2)},$$
Lemma \ref{lemma:main-term} immediately yields 

\begin{cor}\label{cor:expansion-tau}
There exist a constant $C<\infty$ so that for any $z=z_{\la}^{a,\om}
\in\ZZ$ for which $\om$ is in the interior of  $\HH_1^{\si_1}(\bbs)$ there exists a variation $\om^{(\eps)}$ of $\om$ in $\HH_1^{\si_1}(\bbs)$ satisfying \eqref{ass:var} so that 
\beq
\ddepsz E(z_{\la}^{a,\om^{(\eps)}})\geq \norm{\tauS(\om)}_{L^2(S^2)} -C\la^{-2}
\eeq
for the corresponding variation of adapted bubbles (with fixed $\la$ and $a$).
\end{cor}

\subsection{Estimates on the tension and the second variation on $\ZZ$}\label{sec:tension} $ $\\
To prove our main result we furthermore need the following estimates on the scaling of the first and second variation of the energy at points of our adapted bubble set.

\begin{lemma}\label{lemma:tension}
For any $z=\zlar\in \ZZ$ and
$w\in \Gamma^{H^1}(z^*TN)$ with $\norm{w}_z=1$ we can bound 
\beq
\label{est:tension-times-w}
\abs{dE(z)(w)}\leq C\la^{-2}(\log\la)^{\half}+C\norm{\tauS(\om)}_{C^1(S^2)}, 
\eeq
while for all variations 
$z_\eps=z_{\la_\eps}^{a,\bbeps}$
 satisfying \eqref{ass:var} we have
 \beqa \label{est:second-var} 
\abs{d^2E(z)(\peps z ,w)}&\leq C\la^{-2}(\log\la)^\half  +C\norm{\tauS(\om)}_{C^1(S^2)}+C\la^{-1}\norm{\peps \bbeps}_{C^2(S^2)}\\
&+ 
C\norm{P_{\om}(\peps \tauS(\bbeps))}_{L^2(S^2)}.
\eeqa
\end{lemma}

\begin{rmk}\label{rmk:second-var}
For the variations $z_\eps^{(1)}=z_{\la(1-\eps)}^{a, \om}$ considered in Corollary \ref{cor:Expansion-scaling} this lemma yields 
a bound  of
\beq\label{est:second-var-case1}
\norm{d^2E(z)(\peps z_\eps^{(1)}, \cdot)}\leq C \la^{-2}(\log\la)^\half +C\norm{\tauS(\om)}_{C^1(S^2)}
\eeq
where we compute the operator norm with respect to $\norm{\cdot}_z$. 
For more general variations 
 the term $\norm{P_{\om}(\peps \tauS(\bbeps))}_{L^2(S^2)}=\norm{L_{\bb}(\partial \bbeps)}_{L^2(S^2)}$ 
can be of order one,  but will be small since
$T_\bbs \HHone=\ker(L_\bbs)$.
For variations $z_\eps^{(2)}$ as in Corollary \ref{cor:expansion-tau} we shall hence simply use that, after increasing $\la_1$ and decreasing $\si_1>0$ if necessary,
\beq\label{est:second-var-case2}
\norm{d^2E(z)(\peps z^{(2)}, \cdot)}\leq \eta_3
\eeq for a small constant $\eta_3>0$ that is chosen later on.
\end{rmk}

\begin{proof}[Proof of Lemma \ref{lemma:tension}]
The main step in the proof of the lemma is to derive suitable bounds on the tension $\tau_g(z)$ and its variation $P_z(\partial_\eps \tau_g (z))$  on  $B_{\tilde \iota}(a)$. To do this we can work in the usual isothermal coordinates in which 
$z$ is represented by $\tilde z$ and estimate the tension of $\tilde z$ with respect to the Euclidean metric on $\DDr$. Writing $\tilde z$ as in  \eqref{eq:repr-zla-error} gives
\beqa
\tau (\tilde z)&=P_{\tilde z}(\Delta  \tilde z)
=T_1 +T_2 + \text{err}_1 \text{ on } \DDr
\eeqa
for terms 
$$
T_1:=P_{\tilde z}(\Delta  \omlt) \text{ and }T_2:=P_{\tilde z}(\Delta  (P_\omlt j))
$$ that we analyse in detail below and 
an error term $\text{err}_1= P_{\tilde z}(\Delta  E)$ for which 
\eqref{est:Ela-new} gives
 $$\abs{\text{err}_1}+ \abs{\peps \text{err}_1}= O(\la^{-2}).$$
As we can write 
$T_1=(P_{\omlt}-P_{\tilde z}) (A(\omlt)(\na \omlt,\na\omlt))+P_{\tilde z}(\tau (\omlt))$ we can estimate 
\beq \label{est:T1-proof}
\abs{T_1}\leqs \abs{{\tilde z}-\omlt} \rho_\la^2+\abs{\tau (\omlt)}\leqs  \la^{-1}\abs{x}\rho_\la^2+\norm{\tauS (\om)}_{C^0(S^2)} \rho_\la^2,
\eeq
compare \eqref{eq:trafo-tension-la}. 
Furthermore we can use \eqref{est:Ela-new}-\eqref{est:pepsz} to bound 
\beqas
\abs{P_{\tilde z}\partial_\eps T_1}&\leqs \abs{\partial_\eps({\tilde z}-\omlt)}\rho_\la^2+
\abs{{\tilde z}-\omlt}(\abs{\partial_\eps \omlt}\rho_\la^2+\abs{\partial_\eps\na \omlt}\rho_\la)+\abs{\peps \tilde z} \abs{\tau (\omlt)}+\abs{P_{\tilde z}(\peps \tau (\omlt))}
\\
&\leqs  \la^{-1}\abs{x} \rho_\la^2 
+
\norm{\tauS(\om)}_{C^0(S^2)} \rho_\la^2 +\abs{P_{\tilde z}(\peps \tau (\omlt))} . 
\eeqas
To bound the last term we use that \eqref{eq:trafo-tension-la} and \eqref{est:various-bounds} give 
 \beq
 \label{eq:est-var-tens}
 \abs{\peps \tau (\omla)- \half \abs{\na \pi_\la} ^2 (\peps\tau_{g_{S^2}}(\om) )\circ \pi_\la} \leqs \norm{\tau_{g_{S^2}}(\om) }_{C^1(S^2)}\rho_\la^2
\eeq
thus allowing us to bound 
\beqa 
\abs{P_{\tilde z}(\peps \tau (\omlt))}&\leq C\la^{-1}\abs{x} \abs{\peps \tau (\omlt))}+\abs{P_{\omla}(\peps \tau (\omlt)))}
\\ &\leqs \la^{-1}\abs{x}
\rho_\la^2+
\norm{\tau_{g_{S^2}}(\om) }_{C^1(S^2)}\rho_\la^2+ \abs{ (P_{\om}\peps \tau_{g_{S^2}}(\om))\circ \pi_\la}\cdot \abs{ \na \pi_\la} \rho_\la. 
\eeqa
All in all we thus have an estimate of 
\beq
\label{est:T1-var} 
\abs{P_{\tilde z}(\partial_\eps T_1)}\leqs 
\la^{-1}\abs{x} \rho_\la^2 
+
\norm{\tauS(\om)}_{C^1(S^2)} \rho_\la^2 +\abs{ (P_{\om}\peps \tau_{g_{S^2}}(\om))\circ \pi_\la}\cdot \abs{ \na \pi_\la} \rho_\la
.\eeq
Since $j$ is harmonic we have that 
$T_2=-P_{{\tilde z}}(\Delta  (P^\perp_\omlt j))$, so working with respect to a local orthonormal frame $\{\nu^k\}$ of $T^\perp N$ and summing over $k$ we get  
\beqa \label{est:T2-in-proof}
T_2
&=-\langle \nu_\omlt^k,j\rangle   P_{{\tilde z}}(\Delta  \nu_\omlt^k)
-\Delta  (\langle \nu_\omlt^k,j\rangle ) (P_{{\tilde z}}-P_\omlt)(\nu_\omlt^k)-
2 \na ( \langle \nu_\omlt^k,j\rangle) P_{{\tilde z}}(\na \nu_\omlt^k)
\eeqa
allowing us to bound 
\beqa\label{est:T2-proof}
\abs{T_2}&\leqs \abs{j}\rho_\la^2+\abs{\na j}\abs{j}\rho_\la+\rho_\la \abs{\langle \nu_{\tilde \om_\la}^k,\na j\rangle} \leqs \la^{-1}\abs{x}\rho_\la^2+\la^{-1}\rho_\la (1+\la\abs{x})^{-1}
\eeqa
since $j$ maps into $T_{\bb(p^*)}N$ and 
$\abs{\tilde \om_\la-\om(p^*)}\leqs (1+\la\abs{x})^{-1}$. 
 
Furthermore, differentiating \eqref{est:T2-in-proof} with respect to $\eps$ and using \eqref{est:various-bounds}
gives
\beqa \label{est:dla-T2}
\abs{\partial_\eps T_2}&\leqs  \big[\abs{j}+\abs{\peps j}\big] \rho_\la^2+ C\big [(\abs{\na j}+\abs{\peps \na j}\big]\rho_\la\abs{j} \\
& +\abs{\langle \nu^k_\omlt,\na j\rangle} \rho_\la+\abs{\langle \nu^k_\omlt,\peps \na j\rangle}\rho_\la+\abs{\peps \omlt} \abs{\na j} \rho_\la\\
&\leqs \la^{-1}\abs{x}\rho_\la^2 +\la^{-1}\rho_\la \abs{\om_\la-\om(p^*)}+
\la^{-1}(\norm{\peps \bbeps}_{C^1} +(1+\abs{\la x})^{-1}) \rho_\la \\
&\leqs \la^{-1}\abs{x}\rho_\la^2 +
\la^{-1}(\norm{\peps \bbeps}_{C^1} +(1+\abs{\la x})^{-1}) \rho_\la 
\eeqa
where the penultimate step uses  \eqref{est:var-om-eta} as well as that 
$\partial_\la \na j=\la^{-1} \na j\in T_{\om(p^*)}N$ and thus 
$\abs{P_{\om(p^*)}(\peps \na j)}\leqs \la^{-1}\norm{\peps \bbeps}_{C^1}$.  

All in all we thus find that on $\DDr$
\beq
\label{est:tension-pw}
\abs{\tau (\tilde z)}\leqs \big[ \la^{-1}\abs{x}\rho_\la+\la^{-1}
(1+\abs{\la x})^{-1}+\norm{\tauS (\om)}_{C^0(S^2)} \rho_\la] \rho_\la+\la^{-2}.
\eeq
while 
\beqa 
\label{est:var-tension-pw}
\abs{P_{\tilde z}(\peps \tau(\tilde z))}& \leqs \la^{-1} \big[
\abs{x} \rho_\la+ \norm{\peps \bbeps}_{C^1(S^2)} +(1+\abs{\la x})^{-1})\big] \rho_\la +\la^{-2}\\
& +
\norm{\tauS(\om)}_{C^1(S^2)} \rho_\la^2 +\abs{ (P_{\om}\peps \tau_{g_{S^2}}(\om))\circ \pi_\la} \abs{\na \pi_\la} \rho_\la,
\eeqa
where we note that
$\norm{(P_{\om}\peps \tau_{g_{S^2}}(\om))\circ \pi_\la\abs{ \na  \pi_\la}}_{L^2(\DDr)}= \sqrt{2}\norm{(P_{\om}(\peps \tauS( \om))}_{L^2(\pi_\la(\DDr))}$.

As the energy is conformally invariant, $\norm{w}_z=1$ and as $\tau_g(z)$ and $\peps \tau_g(z)$ are of order $O(\la^{-2})$  
on $\Si\setminus B_{\tilde \iota}(a)$, compare \eqref{est:v-Delta-away},
 we hence get from 
 \eqref{est:tension-pw}, \eqref{est:Lp-by-z-norm} and \eqref{est:norms-z} that 
\beqas
\abs{dE(z)(w)}&=\babs{\int_{\Si} \tau_g(z) w \dd v_g}\leq C\la^{-2} \norm{w}_{L^1(\Si)} +
\int_{\DDr} \abs{\tau(\tilde z)} \abs{w\circ F_a^{-1}} dx\\
&\leqs  \la^{-2}(\log\la)^\half+
\la^{-1}\norm{\abs{x}\rho_\la}_{L^2(\DDr)} +\la^{-1}
\norm{(1+\abs{\la x})^{-1}}_{L^2(\DDr)}+\norm{\tau_{g_{S^2}}(\om)}_{C^0(S^2)}\\
&\leqs \la^{-2}(\log\la)^\half+\norm{\tau_{g_{S^2}}(\om)}_{C^0(S^2)}
\eeqas 
as claimed in the lemma. Finally, as $w$ is tangential to $N$ along $z$ we have
\beqas
d^2E(z)(\peps z ,w)&=-\int_\Si \peps\tau_g(z) \cdot w \dd v_g= -\int_\Si P_z(\peps\tau_g(z)) \cdot w\dd v_g\\
&= -\int_{\DDr} P_{\tilde z}(\peps\tau(\tilde z)) \cdot w\circ F_a^{-1} dx+ O(\la^{-2}\norm{w}_{L^1(\Si)})\eeqas
and inserting \eqref{est:var-tension-pw} and 
\eqref{est:norms-z} immediately gives 
the second claim \eqref{est:second-var} of the lemma. 
\end{proof}

\section{Proof of Theorem \ref{thm:main}}
\label{sect:proof-main}
We now turn to the proof of our first main result. To this end we  first observe that two adapted bubbles with quite different scales $\la_1,  \la_2$ respectively with quite different underlying maps $\om_1$ and $\om_2$  cannot be close. Namely we have the following lemma of which we provide a short proof in the appendix

\begin{lemma}
\label{lemma:different-scales}
Let $\bbs$ be any harmonic sphere and let $\la_1\geq 2$, $\si_1>0$ be any given numbers for which $\ZZ=\ZZ_{\la_1}^{\si_1}$ is well defined. Then there exist numbers $\eps_{3}>0$ and 
 $\la_2\geq \la_1$ depending only on 
 $\si_1$, $\bbs$, $(\Si,g)$ and $N$ so that 
 \beq
\label{est:dist-larger}
\norm{\zone-\ztwo}_{\zone}\geq \eps_3 
\eeq
for all elements 
$\zone, \ztwo\in \ZZ$ with $\la\geq \la_2$ 
for which either
$\la\tilde \la^{-1} \notin [\half,2] $ 
or for which $\om\in \HH_1^{\frac13 \si_1}(\bbs)$ while 
$\tilde \om\in  \HH_1^{\si_1}(\bbs)\setminus \HH_1^{\frac23\si_1}(\bbs)$. 

Furthermore, for any 
$\eps>0$  there exist $\de>0$ and $\la_3\geq \la_2$ so that 
for any $\zlar, \ztwo\in \ZZ$ with $\la\geq \la_3$ we have 
\beq
\label{est:L-infty-difference}
\norm{\zone-\ztwo}_{L^\infty(\Si)}< \thalf\eps
\text{
whenever }\norm{\zone-\ztwo}_{\zone}<\de.\eeq
\end{lemma}

This lemma now allows us to prove that in the setting of Theorem \ref{thm:main} we have

\begin{lemma}\label{lemma:choosing-z}
Let $\bbs$ be any harmonic sphere and let $\la_1\geq 2$, $\si_1>0$ be any given numbers for which $\ZZ=\ZZ_{\la_1}^{\si_1}$ is well defined. Then for every $\eps>0$ there exist $\eps_1>0$ and $\bar\la\geq \la_1$  so that for any $u\in H^1(\Si,N)$ for which 
$$\norm{u-z_0}_{L^\infty(\Si)}+\norm{\na (u-z_0)}_{L^2(\Si)}<\eps_1 \text{ for some }z_0\in \ZZ_{\bar \la}^{\frac13 \si_1}$$
we have that the infimum 
$\dist(u,\ZZ):=\inf_{ z\in \ZZ} \norm{u- z}_{ z}$ is attained on $\ZZ$ and for every minimiser $z$ of this distance we have that 
$z\in \ZZ\setminus \partial Z$ and 
\beq 
\label{def:z-inf}
\norm{u-z}_{L^\infty(\Si)}<\eps.
\eeq
\end{lemma}
\begin{proof}[Proof of Lemma \ref{lemma:choosing-z}]
Let $\eps>0$ and let
$u\in H^1(\Si,N)$ and $z_0=z_{\la_0}^{a_0,\om_0}\in \ZZ_{\bar \la}^{\frac13\si_1}$ be so that 
the above assumptions are satisfied for numbers
$\eps_1\in (0,\half \eps)$ and $\bar \la> \max(\la_3,2\la_1)$ that are chosen below, 
$\la_3$ the constant from Lemma \ref{lemma:different-scales}.

Since $\int\rho_z^2 \dd v_{g}$ is bounded uniformly on $\ZZ$, we have that 
$
\norm{u-z_0}_{z}\leq C\eps_1$ for every $z\in\ZZ$.
We can thus choose $\eps_1>0$ small enough so that   
any $z=\zlar\in \ZZ$ with $\norm{ z-u}_{ z}\leq \norm{z_0-u}_{z_0}$ must be so that 
$$\norm{z_0- z}_{z}\leq \norm{z-u}_{z}+\norm{z_0-u}_z\leq \norm{z_0-u}_{z_0}+ \norm{z_0-u}_z\leq 2C\eps_1
<\min(\de,\eps_3),$$ for $\de, \eps_3>0$ as in Lemma \ref{lemma:different-scales}. 
As $\la(z_0)\geq \la_3 \geq \la_2$ we can apply this lemma to conclude that 
any such $z$ is contained in the compact subset of adapted bubbles for which the parameters are constrained by $\la\in[\half\la(z_0),2\la(z_0)]$ and
$\om\in \HH_1^{\frac23\si_1}$. 
Hence 
$z\mapsto \norm{u-z}_{z}$ achieves its minimum over $\ZZ$ on this compact subset and any minimiser in $\ZZ$ is contained in this subset of $\ZZ\setminus \partial Z$. Furthermore, the last part of Lemma \ref{lemma:different-scales} yields that any minimiser satisfies
$\norm{u-z}_{L^\infty}\leq \norm{z-z_0}_{L^{\infty}}+ \norm{u-z_0}_{L^\infty}<\frac{\eps}{2}+\eps_1\leq \eps$.   
\end{proof}
As the norm $\norm{\cdot}_z$ depends on $z$, we cannot expect that the difference $w=u-z$ between $u$ and a minimiser $z$ of $\tilde z\mapsto \norm{u-\tilde z}_{\tilde z}$ 
is orthogonal to $T_z\ZZ$. However, as
we shall see in Lemma \ref{lemma:est-variations} in the appendix, we can bound the variation of the weight $\rho_z$ along any variation $z_\eps\in \ZZ$
by 
\beq
\label{est:var-weight-main}
\norm{\peps \rho_{z}}_{L^2(\Si)}\leq C \norm{\peps z}_{z}.
\eeq
This allows us to 
obtain that $w=u-z$ is almost orthogonal to $T_z\ZZ$ in the sense that

\begin{lemma}
\label{lemma:almost orthogonal}
Let  $u\in H^1(\Si,N)$ and suppose that $z\in \ZZ\setminus \partial \ZZ$ minimises 
$\tilde z\mapsto \norm{u-\tilde z}_{\tilde z}$ on $\ZZ$. Then 
$w:=u-z$
satisfies
\beq \label{est:proj-proj-w}\norm{P^{T_z\ZZ}(P_zw)}_z\leq C\norm{w}_{L^\infty(\Si)}\norm{w}_{z}\eeq
where $P_{z}$ denotes the (pointwise) orthogonal projection from 
$\R^n$ to $T_{z(p)}N$, while $P^{T_z\ZZ}:\Gamma^{H^1}(z^*TN)\to T_{z}\ZZ$ is the $\innz$-orthogonal projection.
\end{lemma}

\begin{proof} 
Given any variation $z_\eps$ of a minimiser $z\in\ZZ\setminus \partial Z$ of $\tilde z\mapsto \norm{u-\tilde z}_{\tilde z}$ we can combine the resulting constraint that
$\ddeps \norm{u-z_{\eps}}_{z_\eps}^2=0$   with \eqref{est:var-weight-main} to conclude that at $\eps=0$ 
\beqas
\langle w,\peps z\rangle_{z}&=\int_\Si \rho_{z} \peps \rho_{z_\eps}\abs{w}^2 \dd v_g 
\leq C \norm{w}_{z} \norm{w}_{L^\infty(\Si)}\norm{\peps \rho_{z_\eps}}_{L^2(\Si)}\\
& \leq 
C \norm{w}_{z}\norm{w}_{L^\infty(\Si)}\norm{\peps z_\eps}_{z}.\\
\eeqas
As $\norm{P_zw-w}_{z}\leq C\norm{w}_{L^\infty}\norm{w}_z$, compare  \eqref{est:w-Pw-znorm} below, we thus obtain
the bound
$$
\abs{\langle P_z w,v\rangle_z} \leq C\norm{w}_{L^\infty}\norm{w}_z \norm{v}_z \text{ for every } v\in T_z \ZZ$$
which is equivalent to the claim \eqref{est:proj-proj-w} of the lemma. 
\end{proof}

This almost orthogonality of $w$ to $T_z\ZZ$ is sufficient to
exploit the uniform definiteness of the second variation orthogonal to $\ZZ$. This is crucial to obtain the following initial estimate on $w$, which will play the role of  \cite[Lemma 2.7]{MRS} in this new setting where we work with a family of distances induced by the norms $\norm{\cdot}_z$ on the infinite dimensional set of maps $H^1(\Si,N)$ rather than in a fixed Hilbert-space.

\begin{lemma}\label{lemma:step-1}
There exists $\eps_2>0$ so that for any  $u\in H^1(\Si,N)$ for which 
\beq
\label{est:norm-less-eps}
\norm{u-z}_{z}=\inf_\ZZ \norm{u-\tilde z}_{\tilde z}<\eps_2 \text{ and } \norm{z-u}_{L^\infty(\Si)}<\eps_2
\eeq
for some  $z=z_{\la}^{a,\om}\in\ZZ\setminus \partial \ZZ$, 
we can bound $w:=u-z$ by
\beq
\label{est:step-1}
\norm{w}_{z}^2\leq C (dE(u)-dE(z))( \tilde w_z)+C\log\la\norm{\tau_g(u)}_{L^2(\Si,g)}\norm{w}_{z}^2
\eeq
and therefore have
\beqa \label{est:w-lemma}
\norm{w}_z& \leqs \norm{\tau_g(u)}_{L^2(\Si,g)}(\log\la)^\half +\la^{-2}(\log\la)^\half+\norm{\tau_{g_{S^2}}(\om)}_{L^2(S^2)}+(\log\la)^2\norm{\tau_g(u)}_{L^2(\Si,g)}^2.
\eeqa
Here $\tilde w_z=
(P^{\VV_z^+}-P^{\VV_z^-})(P_z w)$ is defined using the $\langle\cdot,\cdot\rangle_z$-orthogonal projections from $\Gamma^{H^1}(z^*TN)$ to the subspaces $\VV_{z}^{\pm}$ obtained in Lemma \ref{lemma:2ndvar}.
\end{lemma}

Both for this proof, and in later parts of the proof of Theorem \ref{thm:main}, we consider the maps
\beq
\label{def:ut}
u_t=\pi_N(z+tw) \text{ for } w=u-z \text{ and } t\in [0,1].
\eeq
We note that these maps are well defined if $\eps_2<\de_N$ and will be in $H^2(\Si,N)$ since any map $u\in H^1(\Si,N)$ with $\tau_g(u)\in L^2(\Si)$ is automatically in $H^2(\Si,N)$, see e.g. \cite{Moser, Rupflin-old}. We will also use that  
 $\abs{\na u_t}\leq\abs{\na z}+\abs{\na w}$ and thus that for $v\in H^1(\Si,\R^n)$ 
\beq
\label{est:H1-ut}
\norm{v\abs{\na u_t} }_{L^2(\Si)}^2+
 \norm{P_{u_t}(v)}^2_z\leq C\norm{v}^2_{z}+C\int \abs{\na w}^2\abs{v}^2 \ed v_g.
\eeq
We also remark that for any $s\in[0,1]$ we can write
 $\frac{d}{ds} u_s=d\pi_N(z+sw)(w)=P_{u_s}(w)+\text{err}_s$ for an error term $\text{err}_s\in \Gamma^{H^1}(u_s^*TN)$ that is bounded by 
 \beq
 \label{est:err-new} 
 \abs{\text{err}_s}\leq C\abs{w}^2 \text{ with } \abs{\na\text{err}_s}\leq C\rho_z\abs{w}^2+C\abs{w}\abs{\na w}.
 \eeq Integrating over $s\in[0,1]$ 
and using that also $\abs{(P_{u_t}-P_{u_s})(w)}\leq C\abs{w}^2$ 
 we thus get
\beq
\label{est:w-Pw}
\abs{w-\Pt w}\leq C\abs{w}^2 \text{ while } \abs{\na (w-\Pt w)}\leq C \abs{w}\abs{\na w}+C\abs{w}^2\rho_z
\eeq 
for any $t\in [0,1]$ 
and we will in particular use that
\beq
\label{est:w-Pw-znorm}
\norm{w-\Pt w}_z\leq C\norm{w}_{L^\infty}\norm{w}_z.
\eeq

\begin{proof}[Proof of Lemma \ref{lemma:step-1}]
Let $u$ and $z$ be so that \eqref{est:norm-less-eps} 
is satisfied for a number $\eps_2\in (0,\de_N)$ that is chosen below.
We set $w_z:=P_z(w)$, let $\tilde w_z = (P^{\VV_z^+}-P^{\VV_z^-})(w_z)$ be as in the lemma and note that 
\eqref{est:w-Pw-znorm} implies that  
\beq
\label{est:norms-tilde-w-equiv}
\norm{w-w_z}_{z}\leq C\eps_2 \norm{w}_z \text{ while }\norm{\tilde w_z}_{z}\leq \norm{ w_z}_z\leq C\norm{w}_{z}.\eeq
We can hence combine Lemmas \ref{lemma:definite} and \ref{lemma:almost orthogonal} with \eqref{est:bound-sec-var} to obtain that 
\beqas
d^2E(z)(w_z,\tilde w_z)&=d^2E(z)(P^{\VV_z^+}w_z+P^{\VV_z^-}w_z,
P^{\VV_z^+}w_z-P^{\VV_z^-}w_z)+d^2E(z)(P^{T_z\ZZ}w_z,\tilde w_z)\\
&\geq c_0(\norm{P^{\VV_z^+}w_z}_z^2+\norm{P^{\VV_z^-}w_z}_z^2)-C \norm{P^{T_z\ZZ}w_z}_z\norm{\tilde w_z}_z \\
&\geq c_0(\norm{w_z}_z^2 -\norm{P^{T_z\ZZ}w_z}_z^2)-C\norm{w}_{L^\infty} \norm{w}_z^2 \\
&\geq (c_0(1-C\eps_2^2)-C\eps_2)\norm{w}_z^2 \geq \tfrac{c_0}{2}\norm{w}_z^2
\eeqas
where $c_0>0$ is the constant obtained in Lemma \ref{lemma:definite} and where the last inequality holds after reducing $\eps_2>0$ if necessary.
As
\beqas  
(dE(u) -dE(z))
(\tilde w_z)&=\int_0^1 \frac{d}{dt} \big(dE(u_t)(\tilde w_z)\big) \dd t=\int_0^1 \frac{d}{dt} \big(dE(u_t)(\Pt \tilde w_z)\big) \dd t\\
&= \int_0^1d^2E(u_t)(\tfrac{d}{dt} u_t, \Pt \tilde w_z) +dE(u_t)(\tfrac{d}{dt}\Pt \tilde w_z) \dd t\\
&=\int_0^1  d^2E(u_t)(P_{u_t} w+\text{err}_t, P_{u_t}\tilde w_z)+\int_\Si\na u_t\na (\Pt( \tfrac{d}{dt}\Pt \tilde w_z))\dd v_g 
 \dd t
\eeqas
we thus conclude that
\beqa \label{est:somewhere-in-step-1}
\tfrac{c_0}2 \norm{w}_z^2&\leq  d^2E(z)(w_z,\tilde w_z)\leq (dE(u)-dE(z))( \tilde w_z)+\sup_{[0,1]}T_1+T_2+T_3
\eeqa
for $
T_1:= \abs{d^2E(u_t)(P_{u_t}w,P_{u_t}\tilde w_z)-d^2E(z)(w_z,\tilde w_z)}$, $T_2:
=\int\abs{\na u_t}\abs{\na (\Pt( \tfrac{d}{dt}\Pt \tilde w_z))}$ and $T_3:=\abs{d^2E(u_t)(\text{err}_t, P_{u_t}\tilde w_z)}.$

To bound the first term we apply
\eqref{est:diff-second-var-proj} for $v_1=w$ and $v_2=\tilde w_z$ giving
\beqa
\label{est:proof-step1-2}
T_1&\leq C\int \abs{w}\abs{\na w}\abs{\na  \tilde w_z}+ 
 C\int (\abs{w}\rho_z+\abs{\na w}) (\abs{w}\abs{\na \tilde w_z}+\abs{\na w}\abs{\tilde w_z})\\
&\quad +C\int \abs{ \tilde w_z}\abs{w} (\abs{w}\rho_z^2 +\abs{\na w}\rho_z+\abs{\na w}^2)\\
&\leq  C\norm{w}_{L^\infty} \norm{w}_z^2+C\int \abs{\tilde w_z}\abs{\na w}^2\leq (C\eps_2+\tfrac{c_0}{8}) \norm{w}_z^2 +C\int \abs{\tilde w_z}^2\abs{\na w}^2.
\eeqa
As $\tilde w_z$ is obtained using the non-local projections $P^{\VV^\pm_z}$ we do not have a pointwise bound on $\tilde w_z$. 
Instead we 
use that  
$\abs{\Delta_g w}\leq \abs{\Delta_g u}+\abs{\Delta_g z}\leqs \abs{\tau_g(u)}+\rho_z^2+\abs{\na w}^2
$, compare also \eqref{est:Delta-om-trivial},
to bound 
\beqas
I_1&:= \int \abs{\tilde w_z}^2\abs{\na w}^2= -\int w \abs{\tilde w_z}^2\Delta_g w -2\int (w \na w)\cdot(\tilde w_z  \na \tilde w_z)\\
& \leq C\norm{w}_{L^\infty} \big[
\norm{\tau_g(u)}_{L^2}\norm{\tilde w_z}_{L^4}^2+\norm{w}_{z}^2+I_1\big].
\eeqas 
After possibly reducing $\eps_2>0$ and applying \eqref{est:Lp-by-z-norm} we thus get that
\beq
\label{est:I1}
I_1\leq C\eps_2\big[ \log\la \norm{\tau_g(u)}_{L^2}+1\big]\norm{w}_z^2\eeq
and so obtain from \eqref{est:proof-step1-2} that 
\beqa \label{eq:est-T1-proof}
T_1&\leq (C\eps_2+\tfrac{c_0}8) \norm{w}_z^2+ C\eps_2\log\la\norm{\tau_g(u)}_{L^2} \norm{w}_z^2.
\eeqa 
To bound $T_2$ we write, summing over repeated indices $j$,
\beqa
\label{est:snow}
P_{u_t}(\tfrac{d}{dt}\Pt \tilde w_z)&=P_{u_t}\big(-\tfrac{d}{dt}\big(\langle  \nu^j_{u_t}, \tilde w_z\rangle \nu^j_{u_t}\big)\big)= -\langle  \nu^j_{u_t}, \tilde w_z\rangle P_{u_t}(\tfrac{d}{dt} \nu^j_{u_t})\\
&=
-\langle  \nu^j_{u_t}- \nu^j_{z}, \tilde w_z\rangle P_{u_t}(d\nu^j_{u_t}(P_{u_t}w+\text{err}_t))\eeqa
where we use that $\tilde w_z\in T_zN$ in the last step. We can thus estimate
\beqas
T_2
&\leqs\int \abs{\na u_t}(\rho_z\abs{w}+\abs{\na w}) \abs{w}\abs{\tilde w_z}+\abs{\na u_t} \abs{w}\big(\abs{\na w}\abs{\tilde w_z}+\abs{\na \tilde w_z}\abs{w}\big)\dd v_g\leqs \norm{w}_{L^\infty} \norm{w}_z^2.
\eeqas 
Finally, we can use \eqref{eq:second-var} and  \eqref{est:err-new} to bound also
$$T_3\leqs \norm{\na \text{err}_t}_{L^2}\norm{\na \tilde w_z}_{L^2}+\int \abs{\na u_t}^2 \abs{\text{err}_t} \abs{\tilde w_z}\leqs \norm{w}_{L^\infty}  \norm{w}_z^2 .$$
All in all we thus get that 
$$T_1+T_2+T_3\leq (C\eps_2+\tfrac{c_0}8)\norm{w}_z^2+C\eps_2\log\la\norm{\tau_g(u)}_{L^2}\norm{ w}_z^2.$$ 
Combined with 
 \eqref{est:somewhere-in-step-1} this gives the first claim \eqref{est:step-1} of the lemma provided $\eps_2>0$ is chosen sufficiently small.  

We can then combine \eqref{est:step-1} with the bound on $dE(z)$ obtained in Lemma \ref{lemma:tension} and with \eqref{est:Lp-by-z-norm} to deduce that 
\beqas
\norm{w}_z^2& \leq C\big[\norm{\tau_g(u)}_{L^2} (\log\la)^\half 
+\la^{-2}(\log\la)^\half+\norm{\tau_{g_{S^2}}(\om)}_{L^2(S^2)}\big] 
\norm{\tilde w_z}_z
\\
& 
\qquad +C \norm{\tau_g(u)}_{L^2}^2(\log\la)^2\norm{w}_z
+\norm{w}_z^3.
\eeqas
As we can assume that $\eps_2<\half$, we can absorb the last term into the right hand side and use a final time that $\norm{\tilde w_z}_z\leq C\norm{w}_z$ to obtain the second claim of the lemma. 
\end{proof}

We now want to derive suitable bounds on $\la^{-1}$ and on $ \norm{\tau_{g_{S^2}}(\om)}_{L^2(S^2)}$ in terms of the tension of $u$. 
To this end we will exploit the lower bounds on the variations of the energy in the specific directions of $T_z\ZZ$ obtained in Corollaries \ref{cor:Expansion-scaling} and \ref{cor:expansion-tau}  as well as the bounds on the second variation from Lemma \ref{lemma:tension}, see also Remark \ref{rmk:second-var} .

These results tell us that for $z^{(1)}_\eps:= z_{(1-\eps)\la}^{a,\om}$ 
 \beq
 \label{est:E-pepsz1}
 dE(z)(\peps z^{(1)})\geq  c_1\la^{-2} \text{ and } \norm{d^2 E(z)(\peps z, \cdot)}\leq C\la^{-2}(\log\la)^\half  +C\TTom
 \eeq
 while for the variations $z^{(2)}_\eps:= z_{\la}^{a,\bbeps}$ as considered in Corollary \ref{cor:expansion-tau}
   \beq
 \label{est:E-pepsz2}
 dE(z)(\peps z^{(2)})\geq  \TTom -C\la^{-2} \text{ and } \norm{d^2 E(z)(\peps z, \cdot)}\leq  
\eta_3 
 \eeq
 for a number $\eta_3>0$ that we can still choose and for constants $c_1>0$ and $C<\infty$ that only depend on $N$, $\bbs$ and $(\Si,g)$.

Writing again 
 $u_t=\pi_N(z+tw)$ and $w_z=P_z(w)$ for short, we have that for $i=1,2$ 
\beqa \label{est:step2-first-formula}
dE(z)(\peps z^{(i)})
=dE(u)(\peps z^{(i)})+d^2E(z)(\peps z^{(i)}, w_z)-\int_{0}^1 T_4(t)dt
\eeqa
for 
\beqas
T_4&:= \tfrac{d}{dt}[dE(u_t)(\peps z^{(i)})]-d^2E(z)(w_z,\peps z^{(i)})
\\
&=dE(u_t)(\ddt (P_{u_t}\peps z^{(i)}))+ d^2E(u_t)(\ddt u_t, P_{u_t}\peps z^{(i)})- 
d^2E(z)(P_zw,\peps z^{(i)})\\
&=\int \na u_t \na \big(P_{u_t}(\ddt (P_{u_t}\peps z^{(i)}))\big)+d^2E(u_t)(P_{u_t}w+\text{err}_t,P_{u_t}\peps z^{(i)})-d^2E(z)(P_zw,\peps z^{(i)}).
\eeqas
As  $\abs{\peps z^{(i)}}\leqs 1$ and 
$\abs{\peps \na z^{(i)} }\leqs \rho_z$, we can use
\eqref{est:diff-second-var-proj} to bound 
\beqas
&\babs{d^2E(u_t)( P_{u_t}w,P_{u_t} (\peps z^{(i)})-d^2E(z)(P_{z}w,\peps z^{(i)})}
\leqs \norm{w}_z^2\eeqas
while  \eqref{eq:second-var} and \eqref{est:err-new} ensure that also 
$
\babs{d^2E(u_t)(\text{err}_t,P_{u_t} (\peps z^{(i)}))}
\leqs 
\norm{w}_{z}^2.
$

As
$P_{u_t}\big(\ddt (P_{u_t}(\peps z^{(i)}))=-\sum_j\langle  \nu^j_{u_t}-\nu_z,\peps z^{(i)}\rangle P_{u_t}(d\nu_{u_t}(P_{u_t}w+\text{err}_t))$ we can also bound 
\beqas
\babs{\int \na u_t\na \big(P_{u_t}(\tfrac{d}{dt} P_{u_t}\peps z^{(i)})\big) }
&\leqs \int \big[\rho_z+\abs{\na w}\big]\big[\abs{w}\abs{\peps z^{(i)}}(\abs{\na w}+\rho_z\abs{w})+\abs{w}^2\abs{\peps \na z^{(i)}} \big]\\
&\leqs \norm{w}_z^2
\eeqas
and thus get that $\abs{T_4}\leqs  \norm{w}_z^2$. 

For the variation $z^{(1)}_\eps$ which satisfies \eqref{est:E-pepsz1}, we hence obtain from \eqref{est:step2-first-formula} 
that 
\beqas
c_1\la^{-2} &\leq dE(u)(\peps z^{(1)})+d^2E(z)(\peps z^{(1)}, w_z)+C\norm{w}_{z}^2
\\
&\leq \norm{\tau_g(u)}_{L^2}  \norm{\peps z^{(1)}}_{L^2} +C(\la^{-2}(\log\la)^\half  +\TTom)\norm{w}_z+C\norm{w}_z^2\\
&\leqs  \la^{-1}(\log\la)^{\half}\norm{\tau_g(u)}_{L^2}  + \la^{-4}(\log\la)  +\TTom^2 +\norm{w}_z^2,
\eeqas
where norms are computed over $\Si$ unless stated otherwise and where we use that
$$\norm{\peps z^{(1)}}_{L^2}\leqs  \la^{-1}+\norm{(1+\la\abs{x})^{-1}}_{L^2(\DDr)}\leqs \la^{-1}(\log\la)^\half.$$
Combined with Lemma \ref{lemma:step-1}, and after increasing $\la_1$ if necessary, we hence obtain that 
\beqas\la^{-2}&\leqs \la^{-1}(\log\la)^{\half} \norm{\tau_g(u)}_{L^2}  + \log\la  \norm{\tau_g(u)}_{L^2} ^2 +(\log\la)^4\norm{\tau_g(u)}_{L^2}^4
+\TTom^2 .
\eeqas 
Thus either $\la^{-1}(\log\la)^{-\half}\leq \norm{\tau_g(u)}_{L^2}$ and thus 
$\la^{-1}\leq C\norm{\tau_g(u)}_{L^2}(1+\abs{\log\norm{\tau_g(u)}_{L^2}}^\half)$ or $\la^{-1}\leq C\TTom$, so in any case 
\beq
\label{est:la-by-T}
\la^{-1}\leqs   \TTom+ \norm{\tau_g(u)}_{L^2}(1+\abs{\log\norm{\tau_g(u)}_{L^2}}^\half).
\eeq
On the other hand, applying \eqref{est:step2-first-formula} for the variation $z^{(2)}_\eps$ which satisfies \eqref{est:E-pepsz2} as well as $\norm{\peps z^{(2)}}_{L^2}\leq C$ and using Lemma \ref{lemma:step-1} 
 gives 
\beqas
\TTom&\leqs \la^{-2} +\abs{dE(u)(\peps z^{(2)})}+\abs{d^2E(z)(\peps z^{(2)}, w_z)}+\norm{w}_{z}^2
\\
&\leqs \la^{-2} + \norm{\tau_g(u)}_{L^2}  \norm{\peps z^{(2)}}_{L^2} +\eta_3
\norm{w}_z+\eps_3 \norm{w}_z\\
& \leq \la^{-2} (\log\la)^\half+ \norm{\tau_g(u)}_{L^2} (\log\la)^\half+
(\log\la)^2\norm{\tau_g(u)}_{L^2}^2\\
&\quad  +
 (\eta_3+\eps_3) \TTom .
\eeqas
As we can assume that $\eps_3$ and $\eta_3$ are chosen small enough we thus conclude that 
\beqa
\label{est:T-by-la}
\TTom
&\leqs \la^{-2} (\log\la)^\half+ \norm{\tau_g(u)}_{L^2} (\log\la)^\half +  \norm{\tau_g(u)}_{L^2}^2 (\log\la)^2.
\eeqa
We can thus eliminate $\TTom$ from \eqref{est:la-by-T} and get
\beqs
\la^{-1}\leqs \la^{-2}\log\la+ \norm{\tau_g(u)}_{L^2}\big[ 1+\abs{\log\norm{\tau_g(u)}_{L^2}}^\half+(\log\la)^\half
\big]+\norm{\tau_g(u)}_{L^2}^2(\log\la)^2.
\eeqs
For sufficiently large $\la_1$ we hence obtain our claimed bound   \eqref{claim:lambda} of 
\beq
\label{est:claim-Loj3-proven}
\la^{-1} \leq C \norm{\tau_g(u)}_{L^2(\Si)}\big[ 1+\abs{\log\norm{\tau_g(u)}_{L^2(\Si)}}^\half \big]. 
\eeq
Inserting this back into \eqref{est:T-by-la} and using \eqref{claim:Ck-equiv}
implies that
also 
\beq
\label{est:claim-Loj4-proven}
\norm{\tauS(\om)}_{C^k(S^2)}\leq C \TTom \leq C \norm{\tau_g(u)}_{L^2(\Si)}\big[ 1+\abs{\log\norm{\tau_g(u)}_{L^2(\Si)}}^\half \big] 
\eeq
as asserted in \eqref{claim:tension}. From  Lemma \ref{lemma:step-1}
we hence obtain 
\beq 
\label{est:claim-Loj2-proven}
\dist(u,\ZZ)= \norm{w}_z\leq C \norm{\tau_g(u)}_{L^2(\Si)}\big[ 1+\abs{\log\norm{\tau_g(u)}_{L^2(\Si)}}^\half+(\log \la)^\half \big]. 
\eeq 
We now recall that 
$E(z_{\la}^{a,\om})-E(\om)=O(\la^{-2})$, compare   
Remark \ref{rmk:delta-energy-bubbles}, and that $\abs{E(\om)-E(\hat \om)}$ is controlled by the classical 
\Lojns-Simon inequality \eqref{est:Loj-S2}. Combined with the bound on $dE(z)$ from  Lemma \ref{lemma:tension}
this gives
\beqas
\abs{E(u)-E(\bbs)}&\leq \abs{E(u)-E(z)}+ \abs{E(\om)-E(\hat \om)}+
C\la^{-2}\\
&\leqs \abs{dE(z)(w)}+\int_0^1\abs{dE(z)(w)-dE(u_t)(w+\text{err}_t)}\dd t+\TTom^{\gamma_1} +\la^{-2}\\
&\leqs 
\la^{-2}(\log\la)^\half \norm{w}_z+
\norm{\tauS(\om)}_{C^1(S^2)} \norm{w}_z+
\norm{w}_z^2+\TTom^{\gamma_1}+\la^{-2}
\eeqas
since 
$
\abs{dE(z)(w)-dE(u_t)(w)}=\abs{\int\na z\na (P_zw)-\na u_t\na (P_{u_t}w)\dd v_g}
\leq C\norm{w}_z^2.
$ 

Inserting the bounds \eqref{est:claim-Loj3-proven}, \eqref{est:claim-Loj4-proven} and \eqref{est:claim-Loj2-proven} on 
 $\la$, $\tauS(\om)$ and $\norm{w}_z$
 into this estimate hence yields 
\beqa \label{est:claim-Loj3-proven-new}
\abs{E(\hat \omega)-E(u)}\leqs  \norm{\tau_g(u)}_{L^2(\Si,g)}^{\gamma_1} (1+\abs{\log{\norm{\tau_g(u)}_{L^2(\Si,g)}}}^\half)^{\gamma_1}+ \log\la \norm{\tau_g(u)}_{L^2(\Si,g)}^2.
 \eeqa
 These estimates  \eqref{est:claim-Loj2-proven} and \eqref{est:claim-Loj3-proven-new} immediately yield the claimed estimates \eqref{claim:Loj-1} and \eqref{claim:Loj-2} in the generic case where $\log\la\leq C\log \norm{\tau_g(u)}_{L^2(\Si,g)}$ for some fixed $C$, and hence in particular whenever $\la^{-1}\geq \norm{\tau_g(u)}_{L^2(\Si,g)}^8$. Conversely, in the special case where 
 $\la^{-1}\leq  \norm{\tau_g(u)}_{L^2(\Si,g)}^8$ the behaviour of the map on the bubble region and on the bulk are essentially decoupled, allowing us to obtain improved bounds of 
 \beqa \label{est:claim-Loj-improved}
\abs{E(\hat \omega)-E(u)} + \dist(u,\ZZ)^2\leq C \norm{\tau_g(u)}_{L^2(\Si,g)}^2 
 \eeqa
 through a simple cut-off argument that is carried out in Appendix \ref{appendix:conc}.
 
\section{Proofs of Theorem \ref{thm:bubbling-version} and Corollary \ref{thm:energy-spectrum}}
 \label{section:proofs-bubble-general}
 
 \begin{proof}[Proof of Theorem \ref{thm:bubbling-version}]
Let $(u_n)$ be a sequence of almost harmonic maps which converges to a simple bubble tree as described in the introduction. We let $\la_n$, $a_n$ be parameters so that \eqref{eq:strong-conv-bubble-tree} holds. From the definition of the adapted bubbles we hence obtain that
$$\norm{u_n-z_{\la_n}^{a_n,\bbs}}_{L^\infty(\Si,g)}+\norm{\na(u_n-z_{\la_n}^{a_n, \bbs})}_{L^2(\Si,g)}\to 0 \text{ as } n\to \infty$$
where we work on a fixed fundamental domain of $\Si$ and use Euclidean respectively hyperbolic translations to the origin to get a consistent choice of coordinates $F_{a_n}$ in the definition of the adapted bubbles. 
For sufficiently large $n$ we can thus apply Theorem \ref{thm:main}. This immediately yields the claim  \eqref{claim:Loj-2-bubbling-thm} on the energy $E(u_n)$. It also implies that the bubble scale $\tilde \la_n$ of 
elements $z_n=z_{\tilde \la_n}^{\tilde a_n,\tilde \om_n}\in \ZZ$ which minimise $\tilde z\mapsto \norm{u_n-\tilde z}_z$ is controlled by \eqref{claim:lambda}. As  Lemma \ref{lemma:different-scales} implies that the originally chosen $\la_n$ are so that $\la_n\in [\half\tilde \la_n,2\tilde \la_n]$ we also get the same bound on $\la_n$ and for the rest of the proof we can assume that $\la_n=\tilde \la_n$ and $a_n=\tilde a_n$.

If $\tilde \om_n$ is harmonic, which will always be the case in the integrable setting, we can simply set $\om_n=\tilde\om_n$. 
Otherwise we use that the classical \Lojns-Simon inequality \eqref{est:Loj-S2-dist} implies that there exists a harmonic map $\om_n:S^2\to S^2$ which is $C^k$ close to $\bbs$ so that
\beq
\label{est:om-tilde-om}
\norm{\om_n-\tilde \om_n}_{L^2(S^2)}\leq C\norm{\tau(\tilde \om_n)}_{L^2(S^2)}^{\gamma_2} \leq C \TTn^{\,\gamma_2} \abs{\log{\TTn}}^{\frac{\gamma_2}2}.
\eeq
We note that the same type of estimate also holds for $\norm{\om_n-\tilde \om_n}_{C^1(S^2)}$ since both $\om_n$ and $\tilde\om_n$ are elements of $\HHone$. 
We can also use that
\beqas
\norm{\tilde \omega_n \circ \pi_{\la_n}\circ F_{ a_n}-z_n}_{C^1(B_\iota(a_n))} +\norm{z_n-\tilde \om_n(p^*)}_{C^1(\Si\setminus B_\iota(a_n))}
&\leqs \la_n^{-1} 
,\eeqas
compare \eqref{est:v-hat-u-ball} and \eqref{est:v-away}. 
Combining this with  
\eqref{claim:Loj-1} and the already established bound \eqref{claim:la-bubbling-thm} on the bubble scale
we get that for $r_1<\iota$ and sufficiently large $n$
\beqa
&\norm{\na \big( u_n-\omega_n \circ \pi_{\la_n}\circ F_{a_n}\big)}_{L^2(B_{r_1}(a))}+\norm{\na u_n}_{L^2(\Si\setminus B_{r_1}(a))}\\
 &\qquad\leq C\norm{\om_n-\tilde \om_n}_{C^1(S^2)}+C\la_n^{-1} +\norm{u_n-z_n}_{z_n}
\leq C\TTn^{\,\gamma_2} \abs{\log\TTn}^{\frac{\gamma_2} 2}
\eeqa
for the same exponent $\gamma_2\in (0,1]$ for which \eqref{est:Loj-S2-dist} holds. 

To establish the $L^2$-estimate \eqref{claim:L2-est-thm-1} 
we note that 
\beqs
\norm{z_n-\tilde \om_n(p^*)}_{L^2(\Si)}\leqs 
\la_n^{-1}+\norm{(1+\la_n\abs{x})^{-1}}_{L^2(\DD_{r_0})}\leqs 
 \la_n^{-1}(\log \la_n)^{\half}\leqs  \TTn (\log\TTn),
 \eeqs
 compare \eqref{est:norms-z} and \eqref{claim:la-bubbling-thm}. 
Combined with \eqref{est:om-tilde-om} this gives 
\beqs
\norm{z_n-\om_n(p^*)}_{L^2(\Si)}\leqs  \TTn \abs{\log\TTn}+\norm{\om_n-\tilde \om_n}_{C^0}\leqs  \TTn \abs{\log\TTn}+\TTn^{\,\gamma_2} \abs{\log\TTn}^{\frac {\gamma_2}2}.
\eeqs
Finally, 
 \eqref{claim:L2-est-thm-2} follows from \eqref{claim:Loj-2} and \eqref{est:om-tilde-om} since we have a lower bound  of $\rho_{\la_n}\geq c_\La  \la_n$, $c_\Lambda>0$,  on discs $\DD_{\La\tilde \la_n^{-1}}$.
\end{proof}

\begin{proof}[Proof of Corollary \ref{thm:energy-spectrum}]
Let $N$ be an analytic manifold of any dimension, let $(\Si,g)$ be a closed surface of genus at least $1$ and suppose that there exists an accumulation point  $\bar E<E^*=\min(E_{S^2}^*,2E_{S^2}, E_{(\Si,g)}+E_{S^2})$ of the 
energy spectrum.
Thus there are harmonic maps $u_i:\Si\to N$ with $E(u_i)\neq E(u_j)$ for $i\neq j$ and $E(u_i)\to\bar E$. We note that the maps $u_i$ cannot subconverge smoothly to 
 a harmonic map $u_\infty:\Si\to N$ as Simon's \Loj estimate ensures that all harmonic maps in a neighbourhood of $u_\infty$ have the same energy. 
 Thus the sequence must undergo bubbling: As 
each bubble requires energy of at least  $E_{S^2}$ and as $\bar E<2E_{S^2}$ 
 the corresponding bubble tree cannot contain multiple bubbles. As 
 $\bar E$ is also less than $E_{S^2}+E_{(\Si,g)}$, the base map must furthermore be trivial. Finally the assumption that $\bar E<E_{S^2}^*$ ensures that the bubble $\bbs$ is not branched. 
 We are hence in the setting of Theorem \ref{thm:bubbling-version} and the resulting estimate \eqref{claim:Loj-2-bubbling-thm} implies that $E(u_i)=E(\om)$ for sufficiently large $i$ 
 leading to a contradiction. 
\end{proof}
  
\section{Convergence of harmonic map flow} \label{sect:flow}
\begin{proof}[Proof of Theorem \ref{thm:flow}]
Let $u$ be a solution of the harmonic map flow \eqref{eq:HMF} as considered in Theorem \ref{thm:flow}. 
If there is any sequence $t_n\to \infty$ along which the flow converges strongly in $H^1$ to a (potentially trivial) harmonic map $u_\infty:\Si\to N$ then Simon's results from \cite{Simon}  imply that the flow converges indeed along all $t\to \infty$ to $u_\infty$.

We can thus assume that for every sequence $t_n\to \infty$ with $\norm{\tau_g(u(t_n))}_{L^2}\to 0$ a subsequence of $(u(t_n))$ converges to a non-trivial bubble tree. As the flow is not constant, and thus $E(u(t))<E(u(0))\leq E^*$, we can argue as 
 in the proof of Corollary \ref{thm:energy-spectrum} to conclude that these bubble trees, which might depend on the chosen subsequence, 
  are all simple and that the obtained bubbles $\om$ are all unbranched and have energy $E(\om)=E_\infty:=\lim_{t\to \infty}E(t)$.
  
It is convenient to choose the rescalings in this convergence to a bubble tree to be around centres $a(t)$ and at scales $\la(t)$ which are chosen so that 
$$E(u(t), F_{a(t)}^{-1}(\DD_{\la(t)^{-1}})=
\sup_{a\in\Si}E(u(t), F_{a}^{-1}(\DD_{\la(t)^{-1}}))=\thalf E_{S^2},$$
as this ensures that the obtained bubbles  are contained in a 
compact subset 
$K\subset H^2(\Si,N)$ of harmonic spheres: Indeed the upper bound of $\half E_{S^2}$ on the energy of the maps  $u(t_n)\circ(\pi_{\la(t_n)}\circ F_{a(t_n)})^{-1}:S^2\to N$ on balls with fixed radius gives such an upper bound also for the bubbles, which in turn makes it impossible for a 
sequence of such bubbles $\om_n$ to undergo bubbling itself.

As Theorem \ref{thm:main} is applicable on a suitable $H^1\cap L^\infty$ neighbourhood of each $\bbs\in K$, we can consider a finite cover of $K$ by such neighbourhoods to deduce that there 
exist $\eps, \bar\la, C>0$ and $\gamma_1>1$ so that the \Lojns-estimate 
\beq
\label{est:Loj-for-flow} 
\abs{E(u)-E_\infty}\leq C\norm{\tau_g(u)}_{L^2(\Si, g)}^{\gamma_1} (1+\abs{\log\norm{\tau_g(u)}_{L^2(\Si, g)}}) ^{\frac{\gamma_1}{2}}
\eeq
holds true for every 
$u\in H^1(\Si, N)$ for which there exists $\bbs\in K$, $a\in \Si$ and $\la\geq \bar\la$ with
\beq\label{ass:u-close-flow-proof}
\norm{u-z_{\la}^{a,\hat \om}}_{H^1(\Si,g)}+\norm{u-z_{\la}^{a,\hat \om}}_{L^\infty(\Si,g)}<\eps.\eeq
We note that 
there exists $\de_0>0$ and $T\geq 0$ so that 
\eqref{ass:u-close-flow-proof}, and hence \eqref{est:Loj-for-flow}, holds true for all
$u(t)$ with $t\geq T$ and  $\norm{\tau_g(u(t))}_{L^2}<\de_0$; indeed otherwise there would be $t_n\to \infty$ with $\norm{\tau(u(t_n))}_{L^2}\to 0$ for which $(u(t_n))$ does not have a subsequence converging to a simple bubble tree.

As \eqref{est:Loj-for-flow} is trivially true if  $\norm{\tau_g(u(t))}_{L^2}\geq \de_0$ (after increasing $C$ if necessary) and as we can assume that 
$E(T)-E_\infty\leq\half$ we thus conclude that 
$\Ed(t):=E(u(t))-E_\infty$ satisfies
 \beq \label{est:energy-flow-first-est}
0\leq \Ed(t)^{\frac{2}{\gamma_1}} \abs{\log\Ed(t)}^{-1}\leq C_0 \norm{\tau_g(u(t))}_{L^2(\Si)}^2 
\eeq
for $t\geq T$ and some $C_0>0$ 
and thus 
\beqa\label{est:energy-flow-first-est-2}
-\tfrac{d}{dt}\Ed(t)&=\norm{\tau_g(u(t))}_{L^2(\Si)}^2\geq C_0^{-1}\Ed(t)^{\frac{2}{\gamma_1}} \abs{\log\Ed(t)}^{-1} .\eeqa
We can now proceed as in \cite{Simon} and \cite{Topping-quantisation} to establish the claimed convergence of the flow.

If $\gamma_1=2$ then \eqref{est:energy-flow-first-est-2} implies that 
$$\big(\log\Ed(t)\big)^2\geq 2C_0^{-1}(t-T)+\big(\log\Ed(T)\big)^2 \text{ for } t\geq T,$$
which allows us to conclude that
$E(t)-E_\infty\leq Ce^{-c_1\sqrt{t}}.$

If $\gamma_1\in (1,2)$ then the above estimate implies that $\psi:= \Ed^{-\frac{2-\gamma_1}{\gamma_1}}$ satisfies 
$$\ddt \psi=\tfrac{2-\gamma_1}{\gamma_1} \Ed^{-\frac{2}{\gamma_1}} \norm{\tau_g(u(t))}_{L^2(\Si)}^2 \geq \tfrac{2-\gamma_1}{C_0\gamma_1}  \abs{\log\Ed}^{-1}\geq c (\log\psi)^{-1}
$$
 so we conclude that 
$\psi(t)(\log\psi(t)-1)\geq c (t-T)-\psi(T)(\log\psi(T)-1).$ The resulting bound of 
$\psi(t)\geq \tilde c t(\log t)^{-1}$  for some $\tilde c>0$ 
then gives the claimed bound \eqref{claim:delta-energy-flow-sphere-degenerate}
on the decay of the energy.

Given any $0<\al <\frac{\gamma_1-1}{\gamma_1} $ we now
fix $\be\in (\al, \frac{\gamma_1-1}{\gamma_1})$ and  note that   \eqref{est:energy-flow-first-est} gives 
$$\Ed(t)^{\beta-1}\norm{\tau_g(u)}_{L^2}\geq \Ed(t)^{-( \frac{\gamma_1-1}{\gamma_1}-\be)}\abs{\log\Ed(t)}^{-\half}\geq 1$$ for sufficiently large $t$. 
We hence obtain that 
$$-\frac{d}{dt}\Ed(t)^{\beta}=\beta \Ed^{\beta-1}\norm{\tau_g(u)}_{L^2(\Si)}^2\geq  \beta \norm{\tau_g(u)}_{L^2(\Si)}$$
allowing us to conclude that for sufficiently large 
$t<\tilde t$ 
\beq
\label{est:L2-conv-proof-flow}
\norm{u(t)-u(\tilde t)}_{L^2}
\leq \int_t^{\tilde t} \norm{\tau_g(u(s))}_{L^2}\dd s\leq 
C \Ed(t)^{\beta}.
\eeq
We now fix a sequence $t_n\to \infty$ for which $u(t_n)$ converges to a simple bubble tree and denote by $a\in \Si$ the point at which the corresponding bubble $\om$ forms.

Applying the above estimate for 
$\tilde t=t_n$ and using that $\norm{u(t_n)-\om(p^*)}_{L^2}\to 0$ we get
\beq 
\norm{u(t)-\om(p^*)}_{L^2}\leq C \Ed(t)^{\be}\label{est:L2-proof-flow}
\eeq for all sufficiently large $t$, which in particular implies \eqref{est:L^2-est-flow}.

To show that 
 the point $a$ where the bubble forms is independent of the chosen sequence and that the maps converge 
in $C^k$ away from $a$, we can now follow the argument of \cite{Topping-quantisation} and
 combine \eqref{est:energy-flow-first-est-2} with estimates on the evolution of the energy on fixed size balls as proven in 
 \cite[Lemma 3.3]{Topping-quantisation}
 and the $C^k$ control 
on regions with low energy obtained in \cite{Struwe1985}. 

To be more precise, 
\cite[Lemma 3.10']{Struwe1985}, see also \cite[Lemma 3.2]{Topping-quantisation}, assures that there exists $\eps_1=\eps_1(N)>0$ so that for any 
$\Om\subset \Si$, $r\in (0,\inj(\Si,g))$ and $k\in \N$ there exists a constant $C$ so that following holds true:
For any solution $u$ 
 of \eqref{eq:HMF} which satisfies 
$$\sup_{(x,t)\in \Om\times [t_0,\infty)} E(u(t),B_r(x))<\eps_1$$
we can bound $\norm{u(t)}_{C^k( \Om_{\frac{r}{2}})}\leq C$ for  $t\geq t_0+1$ and $\Om_{\frac{r}{2}}:=\{x\in\Si: \dist(x,\Om)\leq \frac{r}2\}$.  

Let now  $\Om$ be a fixed compact subset of $\Si\setminus \{a\}$. 
As  $u(t_n)\to \om(p^*)$ strongly in $H_{loc}^1(\Si\setminus \{a\})$ along the particular sequence of times $t_n\to \infty$ chosen above, we can choose 
 $r>0$  so that 
 $\sup_{x\in \Om} E(u(t_n),B_{2r}(x))\leq \thalf \eps_1$ for all $n$. 
Lemma 3.3 of 
 \cite{Topping-quantisation} and \eqref{est:L2-conv-proof-flow} then allow us to bound 
$$
E(u(t),B_r(x)) \leq  E(u(t_n),B_{2r}(x)) + Cr^{-1}\int_{t_n}^t\norm{\tau(u(s))}_{L^2} ds\leq \thalf\eps_1+C \Ed(t)^\beta \leq \eps_1
$$
for all  $x\in \Om$ and $t\geq t_n$ for sufficiently large $n$. 

Similarly to the argument in \cite{Topping-quantisation} we then combine the
resulting uniform bounds on $\norm{u(t)}_{C^l(\Om_{\frac{r}{2}})}$, $l\in \N$, from \cite{Struwe1985} mentioned above, with the $L^2$-convergence of the flow using an interpolation argument: To this end we recall the standard interpolation inequality 
\beq
\label{est:interpol-inequ}
\norm{f}_{H^s(\Om)}\leq C \norm{f}_{H^{s-m_1}(\Om_{\frac{r}{2}})}^{\frac{m_2}{m_1+m_2}}\norm{f}_{H^{s+m_2}(\Om_{\frac{r}{2}})}^{\frac{m_1}{m_1+m_2}},  \quad C=C(\Om,r,m_{1,2},s)
\eeq
which 
holds for all $m_{1,2}\in \N$ with $m_1\leq s$ and 
follows inductively from integration by parts. 
Given any $k\in \N$ and any $\de>0$ we can apply this inequality for $m_1=s=k+2$ and $m_2=m_2(k,\de)$ sufficiently large to conclude that, for 
$l=k+2+m_2$,
$$\norm{f}_{H^{k+2}(\Om)}\leq C \norm{f}_{L^2(\Om_{\frac{r}{2}})}^{1-\de} \norm{f}_{H^{l}(\Om_{\frac{r}{2}})}^{\de}\leq C \norm{f}_{L^2(\Si)}^{1-\de} \norm{f}_{C^{l}(\Om_{\frac{r}{2}})}^{\de}.$$
Choosing $\de>0$ so that $(1-\de)\beta\geq \al$ and combining this with \eqref{est:L2-proof-flow} and the uniform $C^l$ bounds on $u(t)$ on $\Om_{\frac{r}2}$ hence gives the final claim of the theorem that 
\beqas 
\norm{u(t)-\om(p^*)}_{C^k(\Om)}&\leq C\norm{u(t)-\om(p^*)}_{H^{k+2}(\Om)}\leq C\norm{u(t)-\om(p^*)}_{L^2(\Si)}^{1-\de}\leq C\Ed(t)^{(1-\de)\beta}\\
&\leq C \Ed(t)^\al.
\eeqas
\end{proof}
\begin{proof}[Proof of Corollary \ref{cor:torus}]
As \cite{Eells-Wood} excludes the existence of a harmonic map of degree $\pm1$ from a torus $(T^2,g)$ to $S^2$ any solution $u(t)$ of 
 harmonic map flow as considered in the Corollary will be non-constant, so have $E(u(t))<12\pi$ for $t>0$, and will need to become singular. Indeed we claim that the constraint on the energy means that 
 a single bubble must form, be it at finite or infinite time, and that this bubble must have the same degree as the original map. 
Indeed, the formation of either two bubbles with degree $\pm 1$ or of a bubble of higher degree would only leave energy less than $4\pi$ but at the same time would leave us with a limiting body map of a non-zero degree which is impossible. 
Similarly, the formation of more than two bubbles or of two bubbles which do not both have degree $\pm 1$ would require initial energy greater than $12\pi$ so is also excluded. Finally, if the degree of the bubble and the degree of the map did not agree then we would end up with a body map of degree $k$ with $\abs{k}\geq 2$ which would need energy at least $8\pi$ leading again to a contradiction. 

If the bubble forms at finite time then Simon's result \cite{Simon} yields exponential convergence to a constant. Conversely, if the singularity forms at infinite time then the proof of Theorem \ref{thm:flow} applies and yields convergence at a rate of $O(e^{-c\sqrt{t}})$ since all Jacobi fields along harmonic maps from $S^2$ to $S^2$ are integrable, see \cite{Gulliver-White}.
\end{proof}

\appendix

\section{Definition of $\HHz(\bbs)$ based on Simon's construction from \cite{Simon}}
\label{sec:Simon}
Here we recall the key elements of the argument of Simon that we need to define the manifold $\HHz(\bbs)$ and to check that it has the properties stated in Lemma \ref{lemma:Simon}. 

Let $\bbs:S^2\to S^2$ be any harmonic sphere and let 
$L:=L_{\bbs}$ be the Jacobi operator along $\bbs$ defined in \eqref{def:Jacobi-2}, where here and in the following we work with the $L^2$ inner product and consider $L$ as an operator on maps $w:S^2\to \R^N$ which are tangential to $N$ along $\bbs$.

The starting point of Simon's argument is that since $L:=L_{\bbs}$ is a self-adjoint Fredholm operator, the linear equation 
$Lu=f$ has a unique solution $u\in \ker(L)^{\perp}$ 
if and only if $f\in \ker(L)^\perp$
and this solution furthermore satisfies 
$\norm{u}_{C^{k+2,\beta}}\leq C_{k,\beta} \norm{f}_{C^{k,\beta}}$ for any  $k\in\N$, $\beta>0$. 

As explained in \cite{Simon} these properties of $L$ ensure that 
$\NN:C^{k+2,\beta}\to C^{k,\beta}$ defined by 
$$\NN(w):= P^{\ker L} (w)+P_{\bbs}(\tau_{g_{S^2}}(\pi_N(\bbs+w)))$$
is so that the inverse function theorem  yields a map $\Psi:\UU_1\subset C^{k,\beta}\to \UU_2\subset C^{k+2,\beta} $ 
between suitable neighbourhoods $\UU_{1,2}$ of $0$ in $\Gamma^{C^{k,\beta}}(\bbs^*TN)$ and  $\Gamma^{C^{k+2,\beta}}(\bbs^*TN)$
 so that  
$\norm{\Psi(f)-\Psi(g)}_{C^{k+2,\beta}}\leqs \norm{f-g}_{C^{k,\beta}}$
and
$$\NN\circ \Psi =\Id_{\UU_1} \text{ and } \Psi\circ \NN=\Id_{\UU_2}.$$
As 
$P^{\ker(L)^\perp}(\NN(v))=P^{\ker(L)^\perp}(P_{\bbs}(\tau_{g_{S^2}}(\pi_N(\bbs+v)))$ we have
$$
P^{\ker(L)^\perp}(P_{\bbs}(\tau_{g_{S^2}}(\pi_N(\bbs+\Psi(w)))=P^{\ker(L)^\perp}(\NN\circ \Psi (w))=P^{\ker(L)^\perp}w.$$
This not only implies that the finite dimensional manifold
$$\HH(\bbs):=\{\pi_N(\bbs+\Psi(w)): w\in \ker L\cap \UU_1\}$$
contains all harmonic maps which are sufficiently close to $\bbs$, but also that
\beq
\label{eq:grad-on-mf}
P_{\bbs}(\tauS(\om))\in \ker(L)
\text{ for every } \om\in \HH(\bbs).
\eeq
From the equivalence of norms on the finite dimensional space $\ker(L)$ and the fact that $\norm{\om-\bbs}_{C^{k+2,\beta}}$ will be small if we work on suitably small neighbourhoods $\UU_{1,2}$ we hence get that for $\om\in \HH(\bbs)$
\beq
\label{est:tension-C1-equiv}
\norm{\tau_{g_{S^2}}(\om)}_{C^k}\leq C \norm{P_{\bbs}\tau_{g_{S^2}}(\om)}_{C^k}
\leq C \norm{P_{\bbs}\tau_{g_{S^2}}(\om)}_{L^2} \leq  C\norm{\tau_{g_{S^2}}(\om)}_{L^2}.\eeq
Let now $\VV_0(\bbs)\subset \ker(L)$ be so that 
$$\ker(L)= \VV_0(\bbs)\oplus \VV_{\text{M\"ob}}(\bbs) \text{  for }\VV_{\text{M\"ob}}(\bbs):=  T_{\bbs}\{\bbs\circ M: M \in \Mob\}$$ 
splits $L^2$ orthogonally and set  
\beq
\label{def:HHz-deg}
\HHz(\bbs):=\{\pi_N(\bbs+\Psi(w)): w\in \VVzero(\bbs)\cap \UU_1\}.
\eeq
This  codimension $6$ 
submanifold  of $\HH(\bbs)\subset C^{k+2,\beta}$ 
clearly satisfies \eqref{claim:split-deg} while \eqref{claim:Ck-equiv} follows from \eqref{est:tension-C1-equiv}.  

Furthermore, for any $\bb=\pi_N(\bbs+\Psi(w_\om))\in \HHz(\bbs)$ with $\tau_{g_{S^2}}(\bb)\neq 0$ we can
split 
$$P_{\bbs}(\tau_{g_{S^2}}(\bb))=P^{\VVzero(\bbs)}(P_{\bbs}\tau_{g_{S^2}}(\bb))+P^{\VVMob(\bbs)}(P_{\bbs}\tau_{g_{S^2}}(\bb))\in \ker(L),
$$ set $\TT:= \norm{\tau_{g_{S^2}}(\bb)}_{L^2}$ and consider (for $\eps$ near $0$)
$$\bbeps:= \pi_N\big(\bbs+\Psi( w_\om- \tfrac{2\eps}{\TT}P^{\VVzero(\bbs)}(P_{\bbs}\tau_{g_{S^2}}(\bb))))\in \HHz(\bbs).$$
 The equivalence of norms on $\ker(L)$ implies that at $\eps=0$
$$\norm{\peps\bb^{(\eps)} }_{C^{k+2,\beta}}\leq C \TT^{-1}\norm{ P^{\VVzero(\bbs)}(P_{\bbs}\tau_{g_{S^2}}(\bb))}_{C^{k,\beta}}
\leq C\TT^{-1}\norm{ P^{\VVzero(\bbs)}(P_{\bbs}\tau_{g_{S^2}}(\bb))}_{L^2}\leq C$$ 
for a constant $C$ that only depends on $\bbs$ and $k$.

The conformal invariance of the energy ensures that 
 $\tau_{g_{S^2}}(\om)$ is $L^2$-orthogonal to $\VVMob(\bb)=T_{\om}\{\om\circ M, M\in \Mob\}$. As $d\Psi(0)\vert_{\ker(L)}=\Id$ we can thus bound 
\beqas
\norm{\tfrac{2}{\TT}\tau_{g_{S^2}}(\bb) +\peps \om^{(\eps)}}_{L^2}
&
=\tfrac{2}{\TT}\norm{P^{\VVzero(\om)}\tau_{g_{S^2}}(\bb)-d\pi_N(\om)(d\Psi(w_\om)(P^{\VVzero(\bbs)}(P_{\bbs}\tauS(\om)))}_{L^2}\\
&\leq C(\norm{\bbs-\om}_{C^1}+\norm{w_\om}_{C^{k,\beta}})\leq 1
\eeqas
provided the neighbourhoods $\UU_{1,2}$ are chosen sufficiently small. 

As a consequence we obtain that 
$$\tfrac{d}{d\eps}E(\bbeps)=-\langle \tau_{g_{S^2}}(\om),\peps \om^{(\eps)}\rangle_{L^2}\geq 
\tfrac{2}{\TT} \norm{\tau_{g_{S^2}}(\om)}_{L^2}^2- \norm{\tau_{g_{S^2}}(\om)}_{L^2}
=
\norm{\tau_{g_{S^2}}(\bb))}_{L^2},
$$
which establishes the final property of the manifold $\HHz(\bbs)$ claimed in Lemma \ref{lemma:Simon}.

\section{Proofs of technical lemmas} \label{appendix:technical}
In this appendix we give the proofs of the auxiliary Lemma \ref{lemma:proj-on} used in the proof of Lemma \ref{lemma:definite}, of the auxiliary Lemma \ref{lemma:different-scales} and of the estimate \eqref{est:var-weight-main} used in the proof of Theorem \ref{thm:main} and of the estimate \eqref{est:MV} that we used throughout the paper.

To prove Lemma \ref{lemma:different-scales} we first show the analogue statement for the maps $\hat z_{\la}^{b,\om}:= \om\circ M_\la^{b}
:S^2 \to N$ where
$ M_{\la}^b(x)=\pi_\la(\pi^{-1}(\cdot)-b)$.

\begin{lemma} \label{lemma:est-different-scales-sphere}
Given any  harmonic sphere $\bbs$ 
and any 
 $\si_1>0$, there exists $\eps_3>0$  so that
\beq \label{est:lower-dist-S2}
\norm{\zbone-\hat z_{\tilde \la}^{\tilde b, \tilde \om}}_{\la,b}\geq  2\eps_3 
\eeq
whenever  $\om,\tilde \om\in \HH_1^{\si_1}(\bbs)$ and $\la,\tilde \la>0$ are  either so that
$\la^{-1}\tilde \la\notin [\half,2]$ or so that $\om\in \HH_1^{\frac13\si_1}(\bbs)$ 
while $\tilde \om\in \HH_1^{\si_1}(\bbs)\setminus \HH_1^{\frac23\si_1}(\bbs)$.

Furthermore,  
given any $\eps>0$ there exists $\de>0$ so that
\beq \label{est:A3}
\norm{\zbone-\zbtwo}_{L^\infty(S^2)}<\tfrac14\eps \text{ whenever }
\norm{ \zbone- \zbtwo}_{\la,b}
<2\de.\eeq
\end{lemma}
Here we consider the norms on $H^1(S^2,\R^n)$ defined by 
\beq
\label{def:inner-prod-sphere}
\norm{v}_{\la,b}:=\int_{S^2}\abs{\na v}^2+\thalf c_\Si\abs{\na M_{\la}^{b}}^2 \abs{v}^2 \dd v_{g{_S^2}}
\eeq
for $c_\gamma$ as in Remark \ref{rmk:inner-prod-S2} and use that these norms satisfy 
\beq
\label{eq:covariance-norms} 
\langle v\circ M_{\la}^{b},\tilde v\circ M_{\la}^{b}\rangle_{\la,b}=\langle v,\tilde v\rangle_{1,0} \qquad v,\tilde v \in H^1(S^2,\R^n).
\eeq
\begin{proof}[Proof of Lemma \ref{lemma:est-different-scales-sphere}]
Thanks to \eqref{eq:covariance-norms} it is enough to consider the case where $\la=1$ and $b=0$ so $\zbone=\om\in \HH_1^{\si_1}(\bbs)$. 
We can then use that the closure of 
$F_1:=\{\zbtwo: \tilde \om\notin \HH_1^{\frac23 \si_1}\}$ respectively 
$F_2:=\{\zbtwo: \tilde \la\notin [\half,2]\}$ in $H^1$ is disjoint from the compact sets 
$\HH_1^{\frac13 \si_1}$ respectively $\HH_1^{\si_1}$. Thus the $H^1$-distance between $F_1$ and $\HH_1^{\frac13 \si_1}$ and between $F_2$ and  $\HH_1^{\si_1}$ is positive  
and which yields the first claim of the lemma. 

As we can assume that  $\de< \eps_3$ it then suffices to prove the second claim for
maps $z_{\tilde \la}^{\tilde a,\tilde \la}$ with $\tilde \la\in [\half,2]$. As such maps satisfy uniform $C^2$ bounds 
we obtain \eqref{est:A3} from
Ehrling's lemma applied to $C^2(S^2,g_{S^2})\subset\subset L^\infty(S^2,g_{S^2})\hookrightarrow L^2(S^2,g_{S^2})$.
\end{proof}

\begin{proof}[Proof of Lemma \ref{lemma:different-scales}] 
As we may assume that $\eps_3<
\min_{\om\in \HHone}\norm{\na \omega}_{L^2(S^2)}$ as well as that $\la_2$ is sufficiently large we have that 
\eqref{est:dist-larger} is trivially true 
if either $\tilde \la\notin [C^{-1}\la,C\la]$ or $d(a, \tilde a)\geq C\la^{-1}$ for a suitably large constant $C$.

In particular we can assume that $d_\Si(a,\tilde a)\leq C\la^{-1}<\frac18\iota$, which ensures that the 
 derivatives of the adapted bubbles 
$ \zone, \ztwo$ 
 respectively of the corresponding harmonic spheres 
 $\hat z_{\la}^{0, \om}$ and $\hat z_{ \tilde \la}^{F_{a}(\tilde a), \om}$
are of order
 $O(\la^{-1})$ outside of $B_{\tilde \iota}(a)$ respectively  $\pi(\DD_{r_0/2})\subset S^2$.
As the functions representing $ \zone, \ztwo$  in the isothermal coordinates on $B_{\tilde \iota}(a)$ agree upto $H^1$-errors of order $O(\la^{-1})$ 
with the functions representing $\hat z_{\la}^{0, \om}$ and $\hat z_{ \tilde \la}^{F_{a}(\tilde a), \om}$ in stereographic coordinates we thus have 
$$\norm{\na \zone-\na \ztwo}_{L^2(\Si)}
=\norm{\na \hat z_{\la}^{0,\om}-\na \hat z_{\tilde \la}^{F_{a}(\tilde a),\tilde \om}}_{L^2(S^2)}+O(\la^{-1}).$$ 
For the torus we immediately get the same type of relationship also for the weighted $L^2$-norms, while for higher genus surfaces we need to take into account an additional error term 
that results from the difference of the weights which will be of order $O(\la^{-1}\log(\la)^{1/2})$ since $\int \rho_\la^2\abs{x}^2 \dd x=O(\la^{-2}\log(\la))$, compare Remark \ref{rmk:inner-prod-S2}.
In both cases we hence get that 
 $$\norm{ \zone-\ztwo}_{\zone}=
 \norm{\hat z_{\la}^{0,\om}-\hat z_{\tilde \la}^{F_{a}(\tilde a),\tilde \om}}_{\la,0}+O(\la^{-1}\log(\la)^\half)$$
and the first claim \eqref{est:dist-larger} of Lemma \ref{lemma:different-scales} follows from the corresponding estimates \eqref{est:lower-dist-S2} of Lemma \ref{lemma:est-different-scales-sphere}. 

Similarly, given $\eps>0$ and choosing $\de>0$ small enough so that \eqref{est:A3}
 holds, it suffices to ensure that $\la_3(\log\la_3)^{-1/2}\geq C(\min(\eps,\de))^{-1}$ for a sufficiently large $C$ to derive the required $L^\infty$-bound 
\eqref{est:L-infty-difference} on the difference of the adapted bubbles from the corresponding property \eqref{est:A3} of the bubbles stated in Lemma \ref{lemma:est-different-scales-sphere}. 
\end{proof}
To prove Lemma \ref{lemma:proj-on} as well as \eqref{est:var-weight-main} we furthermore show
%
%
\begin{lemma}\label{lemma:est-variations}
There exist $C>1$ so that for all smooth 1-parameter families
$(b_\eps)\subset \R^2$, $(a_\eps)\subset \Si$, $(\la_\eps)\subset [\la_1,\infty)$ and $(\om^{(\eps)})\subset \HH_1^{\si_1}(S^2)$ we have that $\hat z_\eps:=\hat z_{\la_\eps}^{b_\eps,\om^{(\eps)}}$ satisfies
\beq\label{est:variations-ueps-la}
C^{-1} \norm{\partial_\eps\hat z_\eps}_{\la_\eps,b_\eps} 
\leq \la_\eps^{-1} \abs{\peps \la_\eps}+
\norm{\peps \om^{(\eps)}}_{C^2}+\la_\eps \abs{\peps a_\eps}
\leq C  \norm{\partial_\eps\hat z_\eps}_{\la_\eps,b_\eps} 
\eeq
while the adapted bubbles $ z_\eps =z_{\la_\eps}^{a_\eps, \om^{(\eps)}}\in \ZZ$ satisfy
\beq \label{est:variations-zeps-la}
C^{-1} \norm{\peps z_\eps}_{z_\eps}\leq \la_\eps^{-1} \abs{\peps \la_\eps}+\norm{\peps \om^{(\eps)}}_{C^2}+\la_\eps \abs{\peps a_\eps}\leq C
\norm{\peps z_\eps}_{z_\eps}
\eeq 
and
\beq
\label{est:evol-weight} 
\norm{\peps \rho_{z_\eps}}_{L^2(\Si)}\leq C \la_\eps^{-1} \abs{\peps \la_\eps}+C\la_\eps \abs{\peps a_\eps}\leq C
\norm{\peps z_\eps}_{z_\eps} .\eeq
\end{lemma}

\begin{proof}[Proof of Lemma \ref{lemma:est-variations}]
For variations of $\om=\hat z_{1}^{0,\om}\in \HHone$ 
the estimate \eqref{est:variations-ueps-la} easily follows as 
$\HHone$ is a compact subset of a finite dimensional submanifold of $C^k(S^2,N)$ which is transversal to the action of the M\"obius transforms. 
We can then consider 
$ \hat z_{ \la_0^{-1}\la_\eps}^{ \la_0 (b-b_0), \om^{(\eps)}}$ and use 
 \eqref{eq:covariance-norms}
 to reduce the proof of \eqref{est:variations-ueps-la} to this special case.

To derive \eqref{est:variations-zeps-la} from \eqref{est:variations-ueps-la} it then suffices to check that we only obtain error terms of lower order when we use the approximations $z_{\la_\eps}^{a_\eps,\om^{(\eps)}} \approx 0$ on $\Si\setminus B_{\tilde \iota}(a_\eps)$
and
$$z_{\la_\eps}^{a_\eps,\om^{(\eps)}}(F_{a_0}^{-1}(x))
=\pi_N\big[
\tilde \om_{\la_\eps}^{(\eps)}(F_{a_\eps,a_0}(x))
+j_{\la_\eps}^{a_\eps, \om^{(\eps)}}(F_{a_\eps,a_0}(x))\big]
\approx \tilde \om_{\la_\eps}^{(\eps)}(x-b_\eps)
\text{ on } \DDr,
$$
where $F_{a_\eps,a_0}:= F_{a_\eps}\circ F_{a_0}^{-1} $ and  
$b_\eps:=-F_{a_\eps,a_0}(0)=-F_{a_\eps}(a_0)$. 

A  short calculation shows that at $\eps=0$  we indeed have 
\beqas 
&\norm{\peps  \hat z_{\la_\eps}^{b_\eps, \om^{(\eps)}}}_{\la_0,0, S^2\setminus \pi(\DD_{\frac{r_0}{2}})}+
\norm{\peps \big[z_{\la_\eps}^{a_\eps,\om^{(\eps)}}\circ F_{a_0}^{-1}-\tilde \om_{\la_\eps}^{(\eps)}(\cdot-b_\eps)\big]}_{\la_0, \DD_{r_0}}
+\norm{\peps z_{\la_\eps}^{a_\eps,\om^{(\eps)}}}_{z, \Si\setminus F_{a_0}^{-1}(\DD_{r_0})}
\\ &
\qquad \leq C\la^{-1} \big(  \la^{-1} \abs{\peps \la_\eps}+\la \abs{\peps a_\eps}+\norm{\peps \om_\eps}_{C^2(S^2)}  \big), \eeqas
so
 \eqref{est:variations-zeps-la} follows from \eqref{est:variations-ueps-la} after  replacing $C$ by $2C$ and after possibly increasing $\la_1$. 
 
We finally recall that the weight does not depend on the underlying map $\om^{(\eps)}\in\HHone$ and  is given by $\rho_{z_{\la}^{a,\om}}(p)=\frac{\la}{1+\la^2 \abs{F_a(p)}^2}$ on $B_{\iota}(a)$ while $\rho_{z_{\la}^{a,\om}}\equiv \frac{\la}{1+\la^2r_0^2}$ elsewhere. The final claim \eqref{est:evol-weight} of the lemma thus follows from \eqref{est:variations-zeps-la} as well as 
$$\norm{\partial_\la \rho_{z_{\la}^{a,\om}}}_{L^2(\Si)}^2= \int_{\DD_{r_0}}\babs{\partial_\la\big(\tfrac{\la}{1+\la^2\abs{x}^2}\big)}^2 \dd x+O(\la^{-4})
 \leq C\la^{-2}$$ 
and
\beqas
\norm{\peps \rho_{z_{\la}^{a_\eps,\om}}}_{L^2(\Si)}^2&= \int_{B_{\iota}(a)} 
\babs{\peps\big(\tfrac{\la}{1+\la^2\abs{F_{a_\eps}(p)}^2}\big)}^2 \dd v_g\\
&\leq C\la^6\norm{\peps F_{ a_\eps}}_{L^\infty(B_\iota)}^2 \int_{B_\iota(a)} \tfrac{\abs{F_a(p)}^2}{(1+\la^2\abs{F_a(p)}^2)^4} dv_g\leq C \la^2 \abs{\peps a_\eps}^2.
\eeqas
\end{proof}

To obtain Lemma \ref{lemma:proj-on} we first note that a short calculation shows that 
for variations with 
$\la^{-1} \abs{\peps \la}+
\norm{\peps \om^{(\eps)}}_{C^2}+\la \abs{\peps a} =O(1)$ and for 
$\mu_\eps=\la_0^{-1}\la_\eps, \quad b_{\eps}=-F_{a_\eps}(a_0)$ we have 
$$\peps z_{\la_\eps}^{a_\eps,\om^{(\eps)}}\circ F_{a_0}^{-1}\circ\pi_{\la_0}^{-1} 
=\peps \hat z_{\mu_\eps}^{\la_0 b_\eps,\om^{(\eps)}}+O(\la^{-1})
$$
on the subsets $\pi_{\la}(\DDr)$ which exhaust $S^2$ as $\la\to \infty$, 
while $\peps (z_{\la_\eps}^{a_\eps,\om^{(\eps)}}-\om^{(\eps)}(p^*))$ is of order $O(\La^{-1})$
in
 $H^1\cap L^\infty(\Si\setminus B_{\La\la^{-1}}(a_0))$. 
 
Let now $\{e_j^{\infty}\}_{j=1}^K$ be an on basis of  $X_{\bbs}$. For any fixed $j$ we consider a variation $(b_\eps, \mu_\eps, \om^{(\eps)})$ of $(0,1,\bbs)$  
so that  $e_j^\infty=\peps \hat z_{\la_\eps}^{b_\eps,\om^{(\eps)}}$ and, for $i$ sufficiently large, corresponding variations $a_i^\eps$ with $\la_i \peps b_\eps=-\peps F_{a_i^\eps}(a_0)$, $\la_i^\eps=\la_i\mu_\eps$ and $\om^{(\eps,i)}$ in $\HHone$ so that $\peps \om^{(\eps,i)}\to \peps \om$ in $C^1(S^2)$.

The above error estimates ensure that the resulting elements 
 $\tilde e_j^i=\ddeps z_{\la^i_\eps}^{a^i_\eps,\om^{(i,\eps)}}$
 of $T_{z_i}\ZZ$
 converge to $e_j^\infty$ in the sense described in the lemma. As we furthermore have that $\langle \tilde e_j^{i}, \tilde e_k^{i}\rangle_{z_i}=\de_{jk}+o(1)$ we can hence we obtain the desired orthonormal basis from Gram-Schmidt orthogonalisation.

For the sake of completeness we finally include 
\begin{proof}[Proof of \eqref{est:MV}] 
We can assume that  $2\la^{-1}\leq r_0$ as the claim is trivially true for $\la$ in a bounded range as
  $\rho_z\geq c\la^{-1}$, $c=c(\Si,g)>0$. 

As $\rho_\la^{-2} dv_{g} \geq c\la^2 dv_{g_E}$ on $\DD_{2\la^{-1}}$,  $c=\frac{1}{25}$, we can bound, writing for short 
$\tilde w=w\circ F_a$ 
$$\norm{w}_{z}^2\geq c \la^2\int_{\DD_{2\la^{-1}}\setminus \DD_{\la^{-1}}}  \abs{\tilde w}^2 \ed x\geq c \la\int_{\la^{-1}}^{2\la^{-1}} \int_{S^1}\abs{\tilde w(re^{i\theta})}^2 \ed\theta \dd r.$$
We can thus choose $r\in [\la^{-1},2\la^{-1}]$ so that 
$\abs{\fint_{S^1} \tilde w(re^{i\th})}\leq C \norm{w}_{z}$
and hence bound
\beqas
\babs{\fint_{S^1} \tilde w(r_0 e^{i\th})\, \dd\th}&
\leq C\norm{w}_z+\fint_{S^1}\int_{r}^{r_0} \abs{\partial_s \tilde w(s e^{i\th})}  \,\dd s \dd\th\\
&\leq C\norm{w}_{z}+C[\log(r_0)-\log(r)]^{\half}\norm{\na w}_{L^2(\Si)}\leq C(\log\la)^\half \norm{w}_z.
\eeqas
As a standard compactness argument gives
$\babs{\fint_{\partial B_{\iota}(a)} w \,\dd S_g-\fint_\Si w \, \dd v_g}\leq C\norm{\na w}_{L^2(\Si,g)}$ for some $C=C(\Si,g)$  we hence obtain the claim \eqref{est:MV} from 
 the above bound on $\fint_{\partial B_{\iota}(a)} w \dd S_g=\fint_{S^1} \tilde w(r_0 e^{i\th}) \dd\th$. 
\end{proof}

\vspace{-\baselineskip}
\section{Improved bounds in special case of highly concentrated maps }\label{appendix:conc}
\vspace{-\baselineskip}
In this appendix we explain how a simple analysis of $u$ on each of $\Om_1:= F_a^{-1}(\DD_{2\lambda^{-1/2}})$ and $\Om_0:= \Si\setminus F_a^{-1}(\DD_{2\lambda^{-1/2}}) $, $a,\la$ so that $\norm{u-z}_z+\norm{u-z}_{L^\infty(\Si)}\leq \eps$ for some $z=z_\la^{a,\om}$, can be used to see that
\beq \label{est:claim-decoupled}
\dist(u,\ZZ)^2+ \abs{E(u)-E(\hat \om)}\leqs \la^{-1/4}+\norm{\tau_g(u)}_{L^2(\Si,g)}^2. \eeq
We stress that in the generic case where $\la^{-1}$ scales similarly as $\norm{\tau_g(u)}_{L^2(\Si,g)}$ 
 this estimate is significantly weaker than 
the bounds \eqref{est:claim-Loj2-proven} and \eqref{est:claim-Loj3-proven-new} obtained in Section \ref{sect:proof-main} and that the resulting estimate $\abs{E(u)-E(\hat \om)}\leqs \norm{\tau_g(u)}_{L^2(\Si,g)}^{1/4}$
does not yield convergence of the harmonic map flow. Conversely, in the special case where 
 $\la^{-1}\leq \norm{\tau_g(u)}_{L^2(\Si,g)}^8$, this bound  
 yields improved \Loj estimates of the form \eqref{est:claim-Loj-improved}.

 \newcommand{\Rea}{\text{Re}}
To explain this argument, we first recall that
Lemma 2.9 of \cite{Topping-quantisation}
implies that $E(u,A)\leqs\la^{-1/2}$ for $A:=  F_a^{-1}(\DD_{4\lambda^{-1/2}}\setminus \DD_{\lambda^{-1/2}})$. 

To analyse $u$ on $\Om_1$ we consider 
$u_1(x):= u(F_a^{-1}(8\la^{-1/2}x))$ on $(\DD_{1/2},g_0:=dx^2)$ and exploit that  
$\norm{\tau_{g_0}(u_1)}_{L^2(\DD_{1/2},g_0)}\leqs \la^{-1/2} \norm{\tau_{g}(u)}_{L^2(\Si,g)} $. 
On $\hat A:=\DD_{3/8}\setminus \DD_{1/4}$ we hence have 
$\int_{\hat A } \abs{\na^2 u_1}^2+\abs{\na u_1}^4\leqs  E(u,A)+\norm{\tau_{g_0}(u_1)}_{L^2(A,g_0)}^2\leqs \la^{-1/2}$ by standard $H^2$ estimates on regions with small energy. Thus
we can cut $u_1$ off to a constant on $\hat A$ and project onto $N$ to obtain a map $u_2:\DD_{1/2}\to \N$  which agrees with $u_1$ on $\DD_{1/4}$, is constant outside of $\DD_{3/8}$ and so that $\norm{\tau_{g_0}(u_1)}_{L^2(\DD_{1/2},g_0)}\leqs \la^{-1/4}$. Viewing $u_2$ as a map from 
$T^2=[-\half,\half]^2/\sim$ which is $L^\infty\cap H^1$ close to $z_{\la^{1/2}/8}^{0,\om}\in \ZZ_{T^2}$ then allows us to apply the estimates  \eqref{est:claim-Loj2-proven} and \eqref{est:claim-Loj3-proven-new} obtained in Section \ref{sect:proof-main}  to deduce that 
$\abs{E(u_2)-E(\hat \om)}+\dist(u_2,\ZZ_{T^2})^2\leqs  
\la^{-\gamma_1/4}\log\la \leqs \la^{-1/4}.$ 
This immediately implies that $\abs{E(u,\Om_1)-E(\hat \om)}\leq \abs{E(u_2)-E(\hat \om)}+E(u_2,\hat A)\leqs  \la^{-1/4}$. As $\dist(u_2,\ZZ_{T^2})$ must be achieved for an element $z_{\la_1}^{\om_1,a_1}=\om_1(\la_1(x-a_1))+O(\la_1^{-1})$ of  $\ZZ_{T^2}$ for which  $\la_1\sim \la^\half$ and 
$\abs{a_1}\leqs \la_1^{-1}$, we furthermore deduce that there is an element $\tilde z=z_{\tilde \la}^{\tilde a,\om_1}$ with $\tilde \la\sim \la_1\la^{1/2}\sim \la$  and $\abs{a-\tilde a}\leqs \la^{-1}$  in our original set $\ZZ$ of adapted bubbles on 
$(\Si,g)$ with $\norm{u-\tilde z}_{\tilde z,\Om_1}^2\leqs \la^{-1/4}$. 
As  $\norm{\tilde z}_{\tilde z,\Om_0}+\norm{\rho_{\tilde z}}_{L^2(\tilde z,\Om_0)}\leqs \la^{-1/2}$ we can also bound 
 $\norm{u-\tilde z}_{\tilde z,\Om_0}^2\leqs \la^{-1/4}+E(u,\Om_0)$, so \eqref{est:claim-decoupled} follows once we check that
 $E(u,\Om_0)\leqs R_\la(u):= \la^{-1/4}+C\norm{\tau_g(u)}_{L^2(\Si,g)}^2$.

To see this we can exploit that $\osc_{\Om_0\cup A} u\leq \osc_{\Om_0\cup A} z+\norm{u-z}_{L^\infty(\Si)}\leq \la^{-1/2}+\eps\leq 2\eps$ is small and that $\norm{u-\bar u_A}_{L^2(A)}\leqs r_A E(u,A)^{1/2}\leqs r_A \la^{-1/4}$, $r_A:=\la^{-1/2}$, $\bar u_A:=\fint_A u$. Cutting $u$ off to $\bar u_A$ using a function $\varphi\in C_c^\infty(\Om_0\cup A)$ with  $\varphi\vert_{\Om_0}\equiv 1$ and $\abs{\na\varphi}^2+\abs{\na^2 \varphi}\leqs r_A^{-2}$  hence yields  
$u_0:=u+\varphi (u- u_A)$ with $u_0\vert_{\Om_0}\equiv u$, $\norm{u_0-\bar u_0}_{L^\infty(\Si)}
\leq 2\eps$, $\bar u_0:=\fint_\Si u_0$, and 
\beqs \norm{\Delta_g u_0}_{L^1(A)}
\leqs \norm{\tau_g(u)}_{L^2(A)} r_A+E(u,A)+\norm{\na u}_{L^2(A)}+r_A^{-1}\norm{u-\bar u_A }_{L^2(A)}\leqs R_\la(u).\eeqs 

For sufficiently $\eps>0$ we hence obtain the required bound $E(u,\Om_0)\leqs R_\la(u)$ from 
\beqas
\int_\Si \abs{\na u_0}^2 &\leq \int_{\Si}\abs{\Delta_g u_0}\abs{u_0-\bar u_0}  \leq C \norm{\tau_g(u)}_{L^2(\Om_0)} \norm{u_0-\bar u_0 }_{L^2(\Si)}+ C \eps ( \int_\Si \abs{\na u_0}^2+R_\la(u)) 
\\
&\leq (\tfrac18+C\eps) \int_\Si \abs{\na u_0}^2 dv_g+CR_\la(u).
\eeqas 

\vspace{-1.5\baselineskip}

M. Rupflin: Mathematical Institute, University of Oxford, Oxford OX2 6GG, UK\\
\textit{melanie.rupflin@maths.ox.ac.uk}

\end{document}